\documentclass[letterpaper,11pt,reqno]{amsart}
\usepackage{amssymb, amsmath, amsfonts, amscd}
\usepackage{mathrsfs, latexsym}
\usepackage[all,ps,cmtip]{xy}
\usepackage{array, longtable}
\usepackage{bm, enumitem}
\usepackage{xcolor}

\title{The Noether inequality for threefolds and three moduli spaces with minimal volumes}
\date{\today} 

\author{Meng Chen}
\address{\rm School of Mathematical Sciences \& Shanghai Center for Mathematical Sciences, Fudan University, Shanghai 200433, China}
\email{mchen@fudan.edu.cn}

\author{Yong Hu}
\address{\rm School of Mathematical Sciences, Shanghai Jiao Tong University, Shanghai 200240, China}
\email{yonghu@sjtu.edu.cn}

\author{Chen Jiang}
\address{\rm Shanghai Center for Mathematical Sciences \& School of Mathematical Sciences, Fudan University, Shanghai 200438, China}
\email{chenjiang@fudan.edu.cn}

\newcommand{\bQ}{{\mathbb Q}}
\newcommand{\bP}{{\mathbb P}}
\newcommand{\roundup}[1]{\lceil{#1}\rceil}
\newcommand{\rounddown}[1]{\lfloor{#1}\rfloor}

\newcommand\Vol{\text{\rm Vol}}

\newcommand\OO{{\mathcal{O}}}

\newcommand{\lsgeq}{\succcurlyeq}

\newcommand{\Mov}{\operatorname{Mov}}

\newcommand{\Proj}{\operatorname{Proj}}
\newcommand{\coeff}{\operatorname{coeff}}
\newcommand{\Image}{\operatorname{Im}}
\newcommand{\Supp}{\operatorname{Supp}}

\newcommand{\Aut}{\operatorname{Aut}}
\newcommand{\wt}{\operatorname{wt}}

\newtheorem{thm}{Theorem}[section]
\newtheorem{lem}[thm]{Lemma}

\newtheorem{prop}[thm]{Proposition}
\newtheorem{claim}[thm]{Claim}

\theoremstyle{definition}

\newtheorem{question}[thm]{Question}

\newtheorem{setting}[thm]{Setting}

\newtheorem{rem}[thm]{Remark}

\theoremstyle{remark}

\begin{document}
\begin{abstract} 
We establish the Noether inequality  
 \[\textrm{Vol}(X)\geq \frac{4}{3}p_g(X)-\frac{10}{3}\]  for all projective $3$-folds $X$ of general type with geometric genus $5\leq p_g(X)\leq 10$ where $\textrm{Vol}(X)$ is the canonical volume. This result resolves all remaining cases of the Noether inequality for $3$-folds.

We further investigate the moduli spaces of canonical $3$-folds with small genera and minimal volumes.
For a $3$-fold of general type with geometric genus $2$ and with minimal canonical volume $\frac{1}{3}$, 
we prove that its canonical model  is a hypersurface of degree $16$ in $\mathbb{P}(1,1,2,3,8)$, which gives an explicit description of its canonical ring. This implies that the coarse moduli space $\mathcal{M}_{\frac{1}{3}, 2}$, parametrizing all canonical $3$-folds with canonical volume $\frac{1}{3}$ and geometric genus $2$, is an irreducible unirational variety of dimension $189$. Parallel studies show that $\mathcal{M}_{1, 3}$ is irreducible, unirational, and  $236$-dimensional, and that $\mathcal{M}_{2, 4}$ is irreducible, unirational, and $270$-dimensional. As being conceived, every member in these 3 families is simply-connected. 
\end{abstract}


\maketitle

\pagestyle{myheadings}
\markboth{\hfill M. Chen, Y. Hu, C. Jiang\hfill}{\hfill The Noether inequality for threefolds and three moduli spaces \hfill}
\numberwithin{equation}{section}


\section{Introduction} 

Within birational geometry, perhaps the last phase of classifying projective varieties is to construct and to study the geometry of the concrete moduli space for each designated class. This is corroborated by the theory for curves and surfaces. 

In higher dimensions, the birational classification of smooth projective varieties of general type splits into two steps: the first step is to find a good representative in each birational equivalence class, namely, the canonical model, and this step is completed by the minimal model program (see \cite{KMM, K-M, BCHM}); the second step is to study the geometry of canonical models and the moduli space of canonical models with given discrete numerical invariants (see \cite{KSB,Vie95,Kollar-moduli}).
For example, we are interested in important invariants such as the geometric genus $p_g$
and the canonical volume $\text{\rm Vol}$. 
While the general moduli theory is well-established, the explicit geometry of certain concrete moduli spaces is not yet well-understood. 
 
For a positive rational number $a$ and a nonnegative integer $b$, denote by $\mathcal{M}_{a,b}$ the coarse moduli space parametrizing all projective $3$-folds $X$ with canonical singularities such that $K_X$ is ample, $\text{\rm Vol}(X)(=K_X^3)=a$, and $p_g(X)=b$ in the sense of Viehweg \cite[\S8.5]{Vie95} and Koll\'ar--Shepherd-Barron \cite[\S5.4]{KSB} (see also \cite{Kollar-moduli}). 
Then there are 2 main problems about $\mathcal{M}_{a,b}$: the first problem is to determine whether $\mathcal{M}_{a,b}$ is non-empty, namely, to determine the restriction on $(a,b)$ so that there exists at least one such projective $3$-fold, this is usually referred to as the geography problem; the second problem is, when $\mathcal{M}_{a,b}$ is nonempty, 
to know its geometric properties as detailed as possible, for example, the number of its irreducible or connected components and the dimension of each component. 
A further task is to construct an ideal compactification for ${\mathcal{M}_{a,b}}$, and one successful way is the KSB or KSBA moduli theory via the minimal model program (see \cite{KSB, Kollar-moduli}). To the best of our knowledge, for the explicit geometry of moduli spaces of $3$-folds of general type, very little is known in the literature. 


In this paper, first we study the geography problem by completely solving the Noether inequality for $3$-folds, then we investigate the birational geometry and moduli spaces of several special families, of $3$-folds of general type, with small genera and attaining minimal canonical volumes. 


\subsection{The Noether inequality for $3$-folds}\
The classical Noether inequality, established by Max Noether \cite{Noether70} over $150$ years ago, says that for any smooth minimal surface $S$ of general type,
\[
K_S^2\geq 2p_g(S)-4. 
\]
It is a milestone in the history of the classification theory of algebraic surfaces. 

To establish the correct form of the Noether inequality in dimension $3$ is a classical and long-standing problem. {The works \cite{Kob, Chen04, CCZ, CC15} treated Gorenstein  minimal 3-folds. For an arbitrary minimal $3$-fold $X$ of general type, it was proved by Chen \cite{Chen07} (in the case $p_g(X)\leq 4$) and by Chen, Chen and Jiang \cite{Noether, Noether_Add} (in the case $p_g(X)\geq 11$) that the 
Noether inequality \[K_X^3\geq \frac{4}{3}p_g(X)-\frac{10}{3}\] holds}. Hu and Zhang \cite{HZ} classified the equality case when $p_g(X)\geq 11$.

In this paper, we completely solve all remaining cases ($5\leq p_g\leq 10$) of the Noether inequality.

\begin{thm}\label{thm: pg 5} Let $X$ be a minimal $3$-fold of general type. Then the 
Noether inequality \[K_X^3\geq \frac{4}{3}p_g(X)-\frac{10}{3}\] holds. Moreover, when $p_g(X)\geq 5$, the equality holds only if the canonical image of $X$ is a surface and $p_g(X)\equiv 1 \bmod 3$. 
\end{thm}
The last assertion partially recovers a result by Hu--Zhang \cite{HZ}. 

{It is known that the most difficult case in proving Theorem~\ref{thm: pg 5} is when $X$ is canonically fibered by $(1,2)$-surfaces (see Proposition~\ref{prop: (1,2) pencil}). We will explain our strategy to overcome this in Remark~\ref{rem proof of noether}}. 
We also get the following result {which gives a better slope for those minimal $3$-folds, of general type, canonically fibered by $(1,2)$-surfaces}. Some more explicit inequalities can be seen in the end of \S\ref{sec: pg5}. 

\begin{thm}[see Theorem~\ref{thm: noether 12 slop}]\label{thm: noether complement} Let $X$ be a minimal $3$-fold of general type with $p_g(X)\geq 11$. {Assume that $|K_X|$ is composed with a pencil of $(1,2)$-surfaces. Then 
\[K_X^3\ge \frac{1}{2}\rounddown{\frac{27p_g(X)-38}{10}}.\]}
\end{thm}

\subsection{Moduli spaces of canonical $3$-folds with small genera and minimal volumes}\

We start by recalling the following theorem:

\begin{thm}[{\cite[Theorem~1.4, Theorem~1.5]{Chen07}}] Let $W$ be a smooth projective $3$-fold of general type. Denote by $\Vol(W)$ the canonical volume of $W$. Then
\begin{enumerate}
\item when $p_g(W)\geq 2$, $\Vol(W)\geq \frac{1}{3}$;
\item when $p_g(W)\geq 3$, $\Vol(W)\geq 1$;
\item when $p_g(W)\geq 4$, $\Vol(W)\geq 2$.
\end{enumerate}
Furthermore, the above three lower bounds $\frac{1}{3}$, $1$ and $2$ are all optimal and the minimal volume is attained only if $p_g=2,3,4$, respectively. 
\end{thm}

It is a natural question to study the classification problem 
for those $3$-folds attaining the above minimal volumes, as analogue to curves with small genus or surfaces with small $c_1^2$. Also, it is expected that the classification of such $3$-folds with extreme numerical values will benefit the development of explicit classification theory of higher dimensional algebraic varieties. 

To study the birational geometry of a variety $W$ of general type, one classical method is to study its {\it canonical ring} $R(W):=\bigoplus_{k\geq 0}H^0(W, kK_W)$ which originates from the work of Zariski and Mumford (\cite{Zariski}). Reid proposed a program to study the canonical ring from the point of view of commutative algebra which predicts the construction of $W$ (as a subvariety in a weighted projective space) from the algebraic structure of $R(W)$, see \cite{Rei79, Rei05}. However, to the best of our knowledge, the explicit algebraic structure of the canonical ring of a variety is only known in the case that the given variety has a very clear geometric structure (e.g. a complete intersection in a weighted projective space). In this paper, we will give explicit descriptions for the canonical rings of smooth projective $3$-folds of general type with $2\leq p_g(X)\leq 4$ attaining minimal volumes.

Our main results are as follows. 

\begin{thm} 
\label{thm: pg=2}
Let $W$ be a smooth projective $3$-fold of general type with $p_g(W)\geq 2$ and $\Vol(W)=\frac{1}{3}$. 
Then 
\begin{enumerate}
 \item $W$ has the same plurigenera as those of a general hypersurface of degree $16$ in $\mathbb{P}(1, 1, 2, 3, 8)$; 
 
 \item the canonical model of $W$ is a hypersurface in $\mathbb{P}(1,1,2,3,8)$ defined by a weighted homogeneous polynomial $f$ of degree $16$, where 
\[f(x_0, \dots,
x_4)=x_4^2+f_0(x_0, x_1, x_2, x_3)\] in suitable homogeneous coordinates $[x_0:\dots:
x_4]$ of $\mathbb{P}(1, 1, 2, 3, 8)$; in other words, the canonical ring of $W$ is 
\[
R(W)\simeq \mathbb{C}[x_0, \dots,
x_4]/(f),
\]
where $\mathbb{C}[x_0, \dots,
x_4]$ is viewed as a weighted polynomial ring with $\wt(x_0, \dots,
x_4)=(1,1,2,3,8)$;

\item $W$ is simply-connected;

\item the moduli space $\mathcal{M}_{\frac{1}{3},2}$ is irreducible, unirational, and $\dim \mathcal{M}_{\frac{1}{3},2}=189$. 
\end{enumerate}
\end{thm}

 \begin{thm}\label{thm: pg=3}
Let $W$ be a smooth projective $3$-fold of general type with $p_g(W)\geq 3$ and $\Vol(W)=1$. 
Then 
\begin{enumerate}
 \item $W$ has the same plurigenera as those of a general hypersurface of degree $12$ in $\mathbb{P}(1, 1, 1, 2, 6)$; 
 
 \item the canonical model of $W$ is a hypersurface in $\mathbb{P}(1,1,1,2,6)$ defined by a weighted homogeneous polynomial $f$ of degree $12$, where 
\[f(x_0, \dots,
x_4)=x_4^2+f_0(x_0, x_1, x_2, x_3)\] in suitable homogeneous coordinates $[x_0:\dots:
x_4]$ of $\mathbb{P}(1, 1, 1, 2, 6)$; in other words, the canonical ring of $W$ is 
\[
R(W)\simeq \mathbb{C}[x_0, \dots,
x_4]/(f),
\]
where $\mathbb{C}[x_0, \dots,
x_4]$ is viewed as a weighted polynomial ring with $\wt(x_0, \dots,
x_4)=(1,1,1,2,6)$;

\item $W$ is simply-connected;

\item the moduli space $\mathcal{M}_{1,3}$ is irreducible, unirational, and $\dim \mathcal{M}_{1,3}=236$. 
\end{enumerate}
\end{thm}

 \begin{thm}\label{thm: pg=4}
Let $W$ be a smooth projective $3$-fold of general type with $p_g(W)\geq 4$ and $\Vol(W)=2$. 
Then 
\begin{enumerate}
 \item $W$ has the same plurigenera as those of a general hypersurface of degree $10$ in $\mathbb{P}(1, 1, 1, 1, 5)$; 
 
 \item the canonical model of $W$ belongs to either of the following $2$ types:
 \begin{enumerate}
 \item {\bf Type I}: a hypersurface in $\mathbb{P}(1,1,1,1,5)$ defined by a weighted homogeneous polynomial $f$ of degree $10$, where 
\[f(x_0, \dots, x_4)=x_4^2+f_0(x_0, x_1, x_2, x_3)\] in suitable homogeneous coordinates $[x_0:\dots: x_4]$ of $\mathbb{P}(1, 1, 1, 1, 5)$; in other words, the canonical ring of $W$ is 
\[
R(W)\simeq \mathbb{C}[x_0, \dots, x_4]/(f),
\]
where $\mathbb{C}[x_0, \dots, x_4]$ is viewed as a weighted polynomial ring with $\wt(x_0, \dots, x_4)=(1,1,1,1,5)$;

 \item {\bf Type II}: a subvariety in $\mathbb{P}(1,1,1,1,2, 5)$ defined by $2$ weighted homogeneous polynomials $q$ and $f$ of degrees $2$ and $10$ respectively, where 
\begin{align*}
 q(x_0, x_1, x_2, x_3){}&=x_3^2+q_0(x_0, x_1, x_2),\\
 f(x_0, \dots, x_5){}&=x_5^2+f_0(x_0, \dots, x_4)
\end{align*} in suitable homogeneous coordinates $[x_0:\dots: x_5]$ of $\mathbb{P}(1, 1, 1, 1, 2, 5)$; in other words, the canonical ring of $W$ is 
\[
R(W)\simeq \mathbb{C}[x_0, \dots, x_5]/(q, f),
\]
where $\mathbb{C}[x_0, \dots, x_5]$ is viewed as a weighted polynomial ring with $\wt(x_0, \dots, x_5)=(1,1,1,1,2,5)$;
 \end{enumerate}
furthermore, a canonical $3$-fold of type II is a specialization of canonical $3$-folds of type I;

\item $W$ is simply-connected;

\item the moduli space $\mathcal{M}_{2,4}$ is irreducible, unirational, and $\dim \mathcal{M}_{2,4}=270$.
\end{enumerate}
\end{thm}

\subsection{Key ideas, methods and the outline}\

We briefly explain our main methods and encountered difficulties while proving these theorems. In the study of explicit geometry of $3$-folds of general type, one of the most challenging cases is when $X$ admits a pencil $|F_X|$ which is not free. 
One key observation (see \S3) is, by applying MMP to a log resolution of the pair $(X,F_X)$ in question, to obtain a relatively minimal model $(X', F)$ over $X$ on which $F$ is free. Denote by $\mu:X'\longrightarrow X$ the birational morphism. Such a birational model can be viewed as a minimal partial resolution for the pencil $|F_X|$, which has two effective applications respectively in proving the Noether inequality and in studying moduli spaces.  

The main difficulty in proving Theorem \ref{thm: pg 5} is that the moving part of $|K_X|$ may have base points. Relation \eqref{eq: K+F>K+F} as well as Lemma \ref{lem: 2K>K} allows us to utilize Catanese--Chen--Zhang \cite[Theorem~3.2]{CCZ}, which helps to recognize more information from both $\mu^*(K_X)|_F$ and $|\roundup{2\mu^*(K_X)}|$ (see details in proofs of Proposition  \ref{prop: (1,2) pencil} and Theorem \ref{thm: noether 12 slop}). This method of estimating lower bound of $K_X^3$ proves to be very successful since almost all lower bounds obtained in the table of Theorem \ref{thm: noether 12 slop} are optimal. Of course, here we only need to treat those minimal 3-folds which are canonically fibered by $(1,2)$-surfaces. 

About the other application of the key model $(X',F)$, let us consider Theorem~\ref{thm: pg=2}, for instance, which is the most typical one. The key part is to show that the canonical model can be embedded into $\mathbb{P}(1,1,2,3,8)$.
As in \cite{330-1, 330-2}, we have an effective criterion on how to embed a given polarized variety into a certain weighted projective space. In the context of Theorem~\ref{thm: pg=2}, to apply this criterion, we need to know all the plurigenera of $W$ and the behavior of pluricanonical systems of $W$. 
With the help of the new birational model $(X',F)$, we can efficiently analyze the geometry of $X$, and get the information on the singularities of $X$. Then we analyze the behavior of pluricanonical systems of $X$ which gives restrictions on the plurigenera of $X$. Together with Reid's Riemann--Roch formula, such restrictions will eventually lead us to the precise value of the holomorphic Euler characteristic of $X$ and all plurigenera of $X$. Hence we can prove that the canonical model of $W$ is a weighted hypersurface. This essentially implies other statements of Theorem \ref{thm: pg=2}.
\bigskip

This paper is organized as follows. In \S\ref{sec: 2}, we recall some basic knowledge and background. In \S\ref{sec: 3}, we construct a relative minimal model resolving a pencil by the minimal model program. In \S\ref{sec: pg5}, we prove the Noether inequality for $5\leq p_g\leq 10$ as an application of the birational model in \S\ref{sec: 3}. In \S\ref{sec: pg2}--\S\ref{sec: pg4}, we study the geometry of minimal $3$-folds of general type with $2\leq p_g\leq 4$ and with minimal canonical volume, such as their singularities, pluricanonical systems and plurigenera.
In \S\ref{sec: wt embed criterion}, we recall the criterion in \cite{330-2} in a slightly more general form. 
In \S\ref{sec: proofs}, we prove our main results on the moduli spaces upon the preparation of previous sections.

\section{Preliminaries}\label{sec: 2}

Throughout this paper, we work over the complex number field $\mathbb{C}$, nevertheless all results can be easily generalized to any algebraically closed field of characteristic $0$.

 We adopt the standard notation and definitions in \cite{KMM, K-M}, and will freely use them.

 A $\bQ$-divisor is said to be {\it $\bQ$-effective} if it is $\bQ$-linearly equivalent to an effective $\bQ$-divisor.
 
\subsection{Singularities}\

A {\it log pair} $(X, B)$ consists of a normal projective variety $X$ and an effective $\mathbb{Q}$-divisor $B$ on $X$ such that $K_X+B$ is $\mathbb{Q}$-Cartier.

Let $f: Y\to X$ be a log
resolution of the log pair $(X, B)$, write
\[
K_Y =f^*(K_X+B)+\sum a_iD_i,
\]
where $\{D_i\}$ are distinct prime divisors and the sum runs through all irreducible components in the union of $f$-exceptional divisors and the strict transform of $B$. The log pair $(X,B)$ is called
\begin{enumerate}[label=(\alph*)]
\item \emph{kawamata log terminal} (\emph{klt},
for short) if $a_i> -1$ for all $i$;

\item \emph{purely log terminal} (\emph{plt}, for
short) if $a_i> -1$ for all $f$-exceptional divisors $F_i$ and all $f$;

\item \emph{terminal} if $a_i> 0$ for all $f$-exceptional divisors $D_i$ and all $f$;
\item \emph{canonical} if $a_i\geq 0$ for all $f$-exceptional divisors $D_i$ and all $f$.
\end{enumerate}
Usually we write $X$ instead of $(X,0)$ in the case $B=0$.

\subsection{Minimal model and canonical model}\label{sec 2.2}\

A normal projective variety $X$ is said to be {\it minimal} if $X$ has $\mathbb{Q}$-factorial terminal singularities and $K_X$ is nef.

A normal projective variety $X$ with canonical singularities is called {\it of general type} if $K_X$ is big. 

According to the minimal model program (see \cite{BCHM}), a smooth projective variety $W$ of general type is always birational to a {\it minimal model} $W_{\min}$ which is a minimal variety of general type, and to a unique {\it canonical model} $W_{\textrm{can}}$ which is a normal projective variety with canonical singularities and $K_{W_{\textrm{can}}}$ is ample.
In this case, the {\it canonical volume} of $W$ is defined to be the self-intersection number of the canonical divisor of a minimal model or the canonical model, i.e., $\Vol(W):=K_{W_{\min}}^{\dim W}=K_{W_{\textrm{can}}}^{\dim W}$. The {\it geometric genus} of $W$ is defined by $p_g(W):=h^0(W, K_W)=h^0(W_{\min}, K_{W_{\min}}).$

In the study of birational geometry of a smooth projective variety of general type, we often freely replace it by its minimal model or canonical model.

For two integers $m>0$ and $n\geq 0$, an {\it $(m, n)$-surface} is a smooth projective surface $S$ of general type with $\Vol(S)=m$ and $p_g(S)=n$.

\subsection{Rational maps defined by linear systems}\label{sec: b setting}\

Let $X$ be a normal projective variety on which $D$ is a $\bQ$-Cartier Weil divisor with $h^0(X, D)\geq 2$.
Take a linear subspace $V\subseteq H^0(X, D)$ with $\dim V\geq 2$ and consider the corresponding sublinear system $\Lambda\subseteq |D|$.
We can consider the rational map defined by $\Lambda$, say
$\Phi_{\Lambda}: X {\dashrightarrow} \bP^{\dim V-1}$, which is
not necessarily well-defined everywhere. By Hironaka's big
theorem, we can take successive blow-ups $\pi: W\to X$ such
that:
\begin{enumerate}[label=(\roman*)]
\item $W$ is smooth and projective;
\item corresponding to $\Lambda$, the sublinear system
$\Lambda_W\subseteq |\rounddown{\pi^*(D)}|$ has free movable part $\Mov(\Lambda_W)$ and, consequently,
the rational map $\gamma=\Phi_{\Lambda}\circ \pi$ is a morphism.
\end{enumerate}
Let $W\overset{\psi}\longrightarrow \Sigma'\overset{s}\longrightarrow \Sigma$
be the Stein factorization of $\gamma$ with $\Sigma=\gamma(W)\subseteq
\bP^{\dim V-1}$. There is a commutative
diagram:
\[\xymatrix@=4em{
W\ar[d]_\pi \ar[dr]^{\gamma} \ar[r]^\psi& \Sigma'\ar[d]^s\\
X \ar@{-->}[r]^{\Phi_{\Lambda}} & \Sigma}
\]

\begin{enumerate}
\item If $\dim(\Sigma')\geq 2$, then a general
member of $\Mov(\Lambda_W)$ is a smooth projective variety by
Bertini's theorem, and is called a {\it general irreducible element} of $\Mov(\Lambda_W)$ (or $\Lambda_W$). 


\item If $\dim(\Sigma')=1$, then 
a general fiber $S$ of $\psi$ is a smooth projective variety
by Bertini's theorem. In this case, we say that
$\Lambda$ {\it is composed with a pencil} and $S$ is called a {\it general irreducible element} of $\Mov(\Lambda_W)$ (or $\Lambda_W$).  
$\Lambda$ is said to be {\it rational} if $\Sigma'\simeq \bP^1$. A general
member of $\Mov(\Lambda_W)$ is of the form
$\sum_{i=1}^a S_{i}$, where $S_{i}$ is a smooth fiber of $\psi$ for each $i$. It is clear that $a= \dim V-1$ if $\Sigma'\simeq \bP^1$, and $a\geq \dim V$ if $\Sigma'\not \simeq \bP^1$. 
\end{enumerate}

We will frequently use the following setting.
\begin{setting}\label{Set up for canonical and pluricanonical maps}
Let $X$ be a minimal $3$-fold of general type with $p_g(X)\ge 2$. Take a birational modification $\pi: W\to X$ as in \S\ref{sec: b setting} with respect to the canonical system $|K_X|$ and keep the same notation. 
Write
\begin{align*}
 K_W{}&=\pi^*(K_X)+E_{\pi},\\
 \pi^*(K_X){}& = M + Z_\pi,
\end{align*}
where $|M|=\Mov |\rounddown{\pi^*(K_X)}|$, $E_{\pi}$ and $Z_\pi$ are effective $\mathbb{Q}$-divisors. 
Let $S$ be a general irreducible element of $|M|$. 
After taking some further blow-up, we may further assume that the following hold:
\begin{enumerate}
 \item $\Mov|K_S|$ is base point free; 
 \item for any integers $0<m\le 4$ and $0\leq j\leq m$, the linear system \[|M_{m, -j}|:= \Mov|\rounddown{\pi^*(mK_X)-jS}|\] is base point free or empty; 
 \item  
$\Supp(Z_\pi+E_\pi)$  
is of simple normal crossing.  
\end{enumerate}
Consider the natural restriction map
\[
\theta_{m,-j}: H^0(W, M_{m,-j})\to H^0(S, M_{m,-j}|_S),
\]
and set $u_{m,-j}=\dim \Image \theta_{m,-j}$.
 
\end{setting}

For two linear systems $\Lambda_1$ and $\Lambda_2$, we write $\Lambda_1 \lsgeq \Lambda_2$ if there exists an effective
divisor $N$ such that $\Lambda_1\supseteq \Lambda_2+N$.

\subsection{Reid's Riemann--Roch formula}\

A {\it basket} is a collection of pairs of integers (permitting weights). 
 
Let $X$ be a projective terminal $3$-fold. 
According to 
Reid \cite{Rei87}, there is a basket of orbifold points (called the {\it Reid basket})
\[B_X=\bigg\{(b_i,r_i)\mid i=1,\cdots, s; 0<b_i\leq \frac{r_i}{2};b_i \text{ is coprime to } r_i\bigg\}\]
associated to $X$, which comes from locally deforming the singularities of $X$ into cyclic quotient singularities, where a pair $(b_i,r_i)$ corresponds to a (virtual) orbifold point $Q_i$ of type $\frac{1}{r_i}(1,-1,b_i)$.

By Reid's Riemann--Roch formula \cite[Corollary~10.3]{Rei87}, for any positive integer $m$, 
\begin{align}
 {}&\chi(X, \OO_X(mK_X))\notag\\
 ={}&\frac{1}{12}m(m-1)(2m-1)K_X^3-(2m-1)\chi(\mathcal{O}_X)+l(m), \label{eq: RR}
\end{align}
where
$l(m)=\sum_i\sum_{j=1}^{m-1}\frac{\overline{jb_i}(r_i-\overline{jb_i})}{2r_i}$ and the first sum runs over all orbifold points in the Reid basket. Here $\overline{jb_i}$ means the smallest nonnegative residue of $jb_i \bmod r_i$.

For a positive integer $m$, the {\it $m$-th plurigenus} of $X$ is denoted by $P_m(X):=h^0(X, mK_X)$, and it can be computed by Reid's Riemann--Roch formula and the Kawamata--Viehweg vanishing theorem when $X$ is minimal of general type and $m\geq 2$. 

\subsection{Hodge index theorem}\ 

We recall the following well-known result which is a consequence of the Hodge index theorem. 
\begin{lem}\label{lem: HIT}
 Let $D_1, D_2$ be $\mathbb{Q}$-Cartier $\mathbb{Q}$-divisors on a normal projective surface $S$. 
\begin{enumerate}
 \item If $D_1, D_2$ are nef, $D_1^2=D_2^2>0$, and $D_1-D_2$ is effective, then $D_1=D_2$.

 \item If $D_1, D_2$ are effective, $(D_1\cdot D_2)=0$, and $D_1+D_2$ is nef and big, then either $D_1=0$ or $D_2=0$.
 
\end{enumerate}
\end{lem}

\begin{proof}
 Recall that the Hodge index theorem states that if $H_1$ and $H_2$ are $\mathbb{Q}$-Cartier $\mathbb{Q}$-divisors on $S$ such that $H_1^2>0$ and 
$(H_1\cdot H_2)=0$, then $H_2^2\leq 0$; moreover, the equality holds if and only if 
$H_2\equiv 0$. Here recall that if $H_2$ is effective, then $H_2\equiv 0$ if and only if $H_2=0$.

(1) Denote $D=D_1-D_2\geq 0$. Since $0=D_1^2-D_2^2=((D_1+D_2)\cdot D)$ and $D_1, D_2$ are nef, we have $(D_1\cdot D)=(D_2\cdot D)=0$. In particular, $D^2=0$. On the other hand, $(D_1+D_2)^2>0$. So by the Hodge index theorem, $D=0$.

(2) Note that $D_1^2+D_2^2=(D_1+D_2)^2>0$. So either $D_1^2>0$ or $D_2^2>0$. Without loss of generality, we may assume that $D_1^2>0$. Then $(D_1\cdot D_2)=0$ implies that $D_2^2\leq 0$ by the Hodge index theorem. On the other hand, $D_2^2=(D_2\cdot (D_1+D_2))\geq 0$ as $D_1+D_2$ is nef. Hence $D_2^2= 0$ and $D_2=0$ by the Hodge index theorem.
\end{proof}

 \section{A relatively minimal model resolving a pencil}\label{sec: 3}
 
In this section, we construct a birational model which resolves a given pencil by using the minimal model program.

 \begin{prop}\label{prop: new model}
Let $X$ be a normal $\mathbb{Q}$-factorial projective variety on which a linear system $\Lambda$ is composed with a pencil. 
 Then there exists a projective variety $X'$ with a birational morphism $g: X'\to X$ having the following properties:
 \begin{enumerate}
 \item $X'$ is $\mathbb{Q}$-factorial terminal;
\item 
 the movable part of the linear system
$\Lambda'$, on $X'$, corresponding to $\Lambda$ is base point free 
and, after taking the Stein factorization, ${\Lambda'}$ induces a fibration $f: X'\to \Gamma$ onto a curve $\Gamma$;

 \item for a general fiber $F$ of $f$, $K_{X'}+F$ is $g$-nef;
 \item $g^*(K_X+ g_*F)-K_{X'}- F$ is an effective $g$-exceptional $\mathbb{Q}$-divisor.
 \end{enumerate}
 \end{prop}
 \begin{proof}
 Take a birational modification $\pi: W\to X$ and a fibration $\psi:W\to \Gamma$ as in \S\ref{sec: b setting} induced by $\Lambda$.
 Take $S_1, S_2$ to be two general fibers of $\psi$, then $S_1$ and $S_2$ are disjoint and smooth. 
The pair $(W, \frac{1}{2} S_1+\frac{1}{2} S_2)$ is klt and terminal. 
By \cite[Corollary~1.4.2]{BCHM}, we may run a $(K_W+ \frac{1}{2} S_1+\frac{1}{2} S_2)$-MMP over $X$ which terminates at a model $g: X'\to X$ such that $K_{X'}+\frac{1}{2} F_1+\frac{1}{2} F_2$ is nef over $X$, where $F_1$ and $F_2$ are strict transforms of $S_1$ and $S_2$ on $X'$. We claim that $X'$ has all required properties.
 
 As $S_1$ and $S_2$ are movable, they are not contracted by the MMP, so $(X', \frac{1}{2} F_1+\frac{1}{2} F_2)$ is terminal by \cite[Corollary~3.42, Corollary~3.43]{K-M}. 
 In particular, $X'$ is $\mathbb{Q}$-factorial terminal.

 We claim that $F_1\cap F_2=\emptyset$. Suppose that $F_1\cap F_2\neq\emptyset$, then there exists a subvariety $Q\subseteq F_1\cap F_2$ of codimension $2$ in $X'$ as $X'$ is $\mathbb{Q}$-factorial. On the other hand, $X'$ is smooth at the generic point of $Q$ by \cite[Corollary~5.18]{K-M}. So by \cite[Lemma~2.29]{K-M}, the blowing up along $Q$ induces an exceptional prime divisor $E$ with discrepancy $a(E, X', \frac{1}{2} F_1+\frac{1}{2} F_2)\leq 1-\frac{1}{2}-\frac{1}{2}= 0,$
 but this contradicts the fact that $(X', \frac{1}{2} F_1+\frac{1}{2} F_2)$ is terminal. So $F_1\cap F_2=\emptyset$.
 
 If $\Gamma\simeq \mathbb{P}^1$, then the free linear system $|F_1|$ induces a natural morphism $f: X'\to \Gamma$. On the other hand, if $g(\Gamma)>0$, then all steps of the MMP are relative over $\Gamma$ since negative extremal rays are generated by rational curves, and hence $W\to \Gamma$ descends to
 a natural morphism $f: X'\to \Gamma$. Clearly in either case, this morphism is defined by the linear system $\Lambda'$, on $X'$, corresponding to $\Lambda$.
 
 Finally consider 
 \[G=g^*(K_X+ g_*F)-K_{X'}- F.\]
 Then $G$ is exceptional over $X$ and \[-G\equiv_X K_{X'}+ F\equiv K_{X'}+\frac{1}{2} F_1+\frac{1}{2} F_2\] is nef over $X$ by our construction. Hence $G\geq 0$ by the negativity lemma \cite[Lemma~3.39]{K-M}.
 \end{proof}

As a consequence of Proposition~\ref{prop: new model}, we will frequently use the following modification for a minimal $3$-fold of general type with nontrivial canonical map.

\begin{setting}\label{set: Special modification}
 
Let $X$ be a minimal $3$-fold of general type with $p_g(X)\ge 2$. Take a general sublinear system $\Lambda\subseteq |K_X|$ which is composed with a pencil. Applying Proposition~\ref{prop: new model} to $\Lambda$, we get a 
 projective birational morphism $\mu: X'\to X$ with a surjective fibration $f: X'\to \Gamma$ 
 such that $X'$ is $\mathbb{Q}$-factorial terminal 
 and
 \begin{align}
 \mu^*(K_X+ F_X)-K_{X'}- F=G\label{eq: K+F>K+F}
 \end{align} 
 is an effective $\mu$-exceptional $\mathbb{Q}$-divisor, where $F$ is a general fiber of $f$ and $F_X=\mu_*F$. 
Here $G$ is independent of $F$ by the negativity lemma \cite[Lemma~3.39]{K-M}. 
We may write
\begin{align*}
K_X{}&=F_X+Z_X,\\
 \mu^*(K_X{})&=F+Z,\\
 K_{X'}{}&=\mu^*(K_X)+E_\mu, 
\end{align*}
 where $Z_X$, $E_\mu$ and $Z$ are effective $\mathbb{Q}$-divisors and $E_\mu$ is $\mu$-exceptional. 
Note that $F$ is a smooth projective surface of general type. Denote by $\sigma: F\to F_0$ the contraction to its minimal model. 
\end{setting}

The following lemma allows us to compare $\mu^*(K_X)|_F$ and $\sigma^*(K_{F_0})$ in an effective way. This is very essential in our proof. Before this, we only have a comparison result up to $\mathbb{Q}$-linear equivalence (see \cite[Corollary~2.3]{Noether}).

\begin{lem}\label{lem: 2K>K}
 Keep the notation in Setting~\ref{set: Special modification}.  
Suppose that $|K_X|$ is composed with a rational pencil, then \[p_g(X)\mu^*(K_X)|_F-(p_g(X)-1)\sigma^*(K_{F_0})\] is effective. 
More precisely, there exists an effective divisor $D\in |K_X|$ and an effective divisor $D_{0}\in |K_{F_0}|$ such that \[p_g(X)\mu^*(D)|_F-(p_g(X)-1)\sigma^*(D_0)\] is an effective $\mathbb{Q}$-divisor.
\end{lem}
\begin{proof}
To simplify the notation we denote $d=p_g(X)-1$.
Clearly $\Lambda$ and $|K_X|$ are composed with the same pencil. 
So we may write $K_X =dF_X+\Delta$, where $\Delta$ is the fixed part of $|K_X|$. 
Choose a general fiber $F'$ of $f$ different from $F$, denote $F'_X=\mu_*F'$. Then
we can take $D=dF'_X+\Delta$. 

We first explain the construction of $D_0$. 
As $D\in |K_X|$, we have $\mu^*(D)+E_\mu\in |K_{X'}|$. Then $\mu^*(D)|_F+E_\mu|_F\in |K_{F}|$, where the restriction is well-defined as $F\not\subset\Supp(\mu^*(D)+E_\mu)$. Then we just take $D_0:=\sigma_*(\mu^*(D)|_F+E_\mu|_F)\in |K_{F_0}|$.
Clearly $\mu^*(D)|_F+E_\mu|_F-\sigma^*(D_0)=K_F-\sigma^*(K_{F_0})$ is effective.

Applying \eqref{eq: K+F>K+F} to $D\in |K_X|$ and $F'$, one has
\[(d+1)\mu^*(D)\geq d\mu^*(D+F'_X)=d(\mu^*(D)+E_\mu+F'+G)\geq d(\mu^*(D)+E_\mu).\]
Then we get the conclusion by restricting on $F$.  
\end{proof}

The following lemma will be applied to study the singularities and Gorenstein index of $X$ 
when the geometry of $F$ is clear.

\begin{lem}\label{lem: plt}
 Keep the notation in Setting~\ref{set: Special modification}. Suppose that $(F, G|_F)$ is klt. 
 \begin{enumerate}
 \item Then $(X, F_X)$ is plt, $F_X$ is normal and klt, and \[K_F+G|_F=(\mu|_F)^*(K_{F_X}).\]

 \item Suppose that $Z_X=0$. For a non-Gorenstein singularity $P$ of $X$, denote by $r_P$ the Cartier index of $K_X$ at $P$. Then $P\in F_X$ and $r_P$ equals to the Cartier index of $K_X|_{F_X}$ (as a well-defined Weil divisor on $F_X$) at $P$. Furthermore, if $P$ is a Gorenstein point on $F_X$, then $r_P=2$.
 
 \end{enumerate}
\end{lem}
\begin{proof}
(1) Recall that by \eqref{eq: K+F>K+F}, 
\[
K_{X'}+F+G=\mu^*(K_X+F_X).
\] 
Since $({F}, G|_{F})$ is klt, by \cite[Theorem~5.50]{K-M}, $(X', F+G)$ is plt in a neighborhood of $F$. Then by \cite[Theorem~5.48]{K-M}, $(X, F_X)$ is plt in a neighborhood of $F_X$. This implies that $(X, F_X)$ is plt as $X$ is terminal. 
In particular, $F_X$ is normal by \cite[Proposition~5.51]{K-M}. The remaining assertions follow from the adjunction formula $(K_X+F_X)|_{F_X}=K_{F_X}$ and \cite[Theorem~5.50]{K-M} as $X$ has only isolated singularities. 

(2) Since $K_X\sim F_X$, all non-Gorenstein singularities of $X$ are contained in $F_X$. Recall that $K_X|_{F_X}$ is a well-defined Weil divisor on $F_X$ since $X$ and $F_X$ both have only isolated singularities and $F_X$ is general. Then $r_P$ equals to the Cartier index of $K_X|_{F_X}$ at $P$ by \cite[Theorem~A.1]{HLS}.
For the last assertion, note that 
$2K_X|_{F_X}\sim (K_X+F_X)|_{F_X}=K_{F_X}$ by the adjunction formula.
So if $K_{F_X}$ is Cartier at $P$, then $r_P=2$. 
\end{proof}

\begin{lem}\label{lem: Z=0}
 Keep the notation in Setting~\ref{set: Special modification}. Then $Z_X=0$ if and only if $\mu^*(Z_X)|_F\equiv 0$.
\end{lem}

\begin{proof}
We only need to show the ``if" part. 
 Take a general very ample divisor $H$ on $X$, by the projection formula, 
 \[(F_X|_H\cdot Z_X|_H)=(H\cdot Z_X\cdot F_X)=(\mu^*(H)\cdot \mu^*(Z_X)\cdot F)=0.\] On the other hand, $K_X|_H=F_X|_H+Z_X|_H$ is a nef and big divisor on $H$. 
Clearly $F_X|_H\neq 0$,
 so $Z_X|_H=0$ by Lemma~\ref{lem: HIT}(2), which implies that $Z_X=0$.
\end{proof}

 \section{The Noether inequality for $3$-folds}\label{sec: pg5}

Throughout this section, we always assume that $X$ is a minimal $3$-fold of general type with $p_g(X)\geq 2$.
\medskip

\begin{setting}\label{setting: pg=5}
 Let $\pi: W\to X$ be as in \S\ref{sec: b setting} with respect to $|K_X|$ and keep the same notation. We have the following commutative
diagram:
\[\xymatrix@=4em{
W\ar[d]_\pi \ar[dr]^{\gamma} \ar[r]^\psi& \Sigma'\ar[d]^s\\
X \ar@{-->}[r]^{\Phi_{|K_X|}} & \Sigma}
\]
We may write 
\begin{align*}
 K_W {}&=\pi^*(K_X)+E_\pi,\\ 
 \pi^*(K_X){}&=M+Z_\pi, 
\end{align*}
where $|M|=\Mov|\rounddown{\pi^*(K_X)}|$, $E_\pi, Z_\pi$ are effective $\mathbb{Q}$-divisor.  

Let $S$ be a general irreducible element of $|M|$. Then $S$ is a smooth projective surface of general type. Denote by $\sigma_S: S\to S_0$ the contraction to its minimal model. 

After taking some further blow-up, we may further assume that the following hold:
\begin{enumerate}
 \item $\Mov|K_S|$ is base point free; 
 
 \item $\Supp(Z_\pi+E_\pi)$ is of simple normal crossing;

\item $\pi$ factorizes through the modification $\mu: X'\to X$ in Setting~\ref{set: Special modification}. 
\end{enumerate}

\end{setting}
 
We first consider the case when $|K_X|$ is not composed with a pencil of $(1,2)$-surfaces.
\begin{prop}\label{prop: not (1,2)}
 Let $X$ be a minimal $3$-fold of general type with $p_g(X)\geq 5$. Suppose that $|K_X|$ is not composed with a pencil of $(1,2)$-surfaces. Then 
\[K_X^3\ge \frac{1}{4}\roundup{\frac{8}{3}(2p_g(X)-5)}.\]
Moreover, the equality holds only if the canonical image is a surface. 
\end{prop}
\begin{proof}
Keep the notation in Setting~\ref{setting: pg=5}.

 If $\dim\Sigma=3$, then by \cite[Theorem~2.4]{Kob}, we have
 $K_X^3\ge 2p_g(X)-6.$ 
 
 If $\dim\Sigma=1$, since $|K_X|$ is not composed with a pencil of $(1,2)$-surfaces, then by \cite[Theorem~4.4, Theorem~4.5]{Noether}, we have $K_X^3>2p_g(X)-6.$ 

If $\dim\Sigma=2$ and a general fiber $C$ of $\psi$ is of genus at least $3$, then by \cite[Theorem~4.1]{Noether}, we have 
$ K_X^3\ge 2p_g(X)-4$.

 It is easy to see that $2p_g(X)-6$ is strictly larger than the lower bound we want. 

 From now on, suppose that 
 $\dim\Sigma=2$ and a general fiber $C$ of $\psi$ is of genus $2$. 
 Since $\dim\Sigma=2$, we may write
 $S|_S\equiv dC$,
 where $d=\deg(s)\deg(\Sigma)\ge p_g(X)-2$.
 By the adjunction formula, 
 \begin{align*}
 K_S&=(K_{X'}+S)|_S=(\pi^*(K_X)+S+E_{\pi})|_S\\
 &=2S|_S+(E_{\pi}+Z_\pi)|_S\equiv 2dC+(E_{\pi}+Z_\pi)|_S.
 \end{align*}
Note that 
\[p_g(S)\geq h^0(W, K_W)-h^0(W, K_W-S)=p_g(X)-1\geq 3\]
as $S\in \Mov|K_W|$. 
 By \cite[Proposition~2.9]{Noether}, 
$ K_{S_0}^2\ge \frac{8}{3}(2d-1)$. 
Since $K_{S_0}^2$ is an integer, we deduce that \[K_{S_0}^2\ge \roundup{\frac{8}{3}(2p_g(X)-5)}.\] By \cite[Corollary~2.3]{Noether} (take $D=K_X$ and $\lambda=1$), 
 $\pi^*(K_X)|_S- \frac{1}{2}\sigma_S^*(K_{S_0})$
is $\mathbb{Q}$-effective. Hence 
\begin{align*}
 K_X^3\ge (\pi^*(K_X)|_S)^2
 \ge \frac{1}{4}K_{S_0}^2\ge \frac{1}{4} \roundup{\frac{8}{3}(2p_g(X)-5)}. 
\end{align*}
The proof is completed.
\end{proof}

From now on, we consider the case that $|K_X|$ is composed with a pencil of $(1,2)$-surfaces. 
The following is the key proposition of this section, which deals with the most difficult case. The first part gives a new proof of \cite[Theorem~4.7]{Noether} with a much better lower bound.

\begin{prop}\label{prop: (1,2) pencil}
 Let $X$ be a minimal $3$-fold of general type with $p_g(X)\geq 3$. Suppose that $|K_X|$ is composed with a pencil of $(1,2)$-surfaces. 

 \begin{enumerate}
 \item If there is a minimal $3$-fold $X_0$, birational to $X$, such that $\text{\rm Mov}|K_{X_0}|$ is base point free, then 
\[K_X^3\ge \frac{1}{2}\rounddown{\frac{27p_g(X)-38}{10}}.\]
 \item If there is no such $X_0$ as in (1), then
\[
 K_X^3\ge \frac{m_1-2}{m_1}\left(\frac{p_g(X)-1}{p_g(X)}+\frac{m_1-2}{2}\right)+\frac{(p_g(X)-1)^2}{p_g(X)}.
\] 
where $m_1=\rounddown{\frac{7p_g(X)-1}{10}}$.
 
 \end{enumerate}

\end{prop}

\begin{rem}\label{rem proof of noether}
 Here we briefly explain our strategy. In order to get a good estimate of $K_X^3$ in this case, we need to get a good estimate of $(\pi^*(K_X)|_S)^2$ and also need to have a good comparison of $K_X^3$ 
 with $(\pi^*(K_X)|_S)^2$. To this end, we consider the movable linear system \[|M_{m, -n}|:= \Mov|\rounddown{\pi^*(mK_X)-nS}|\]
 for various $m,n$. 
First, we consider $n=0$ and $m$ as large as possible such that we have a decomposition
\[ \pi^*(mK_X)|_S \geq H+(m-2)C
\]
where $|H|\neq \Mov|K_S|$ is a non-trivial movable linear system on $S$ and $C\in \Mov|K_S|$. Such a decomposition will give us a good estimate of $(\pi^*(K_X)|_S)^2$; second, we consider $m=2$ and $n$ as large as possible 
 such that we have a decomposition
\[ \pi^*(2K_X) \geq nS+M_{2, -n}
\]
and $|M_{2, -n}|_S|\neq \Mov|K_S|$ is a non-trivial movable linear system on $S$. Such a decomposition will give us a good comparison of $K_X^3$ 
 with $(\pi^*(K_X)|_S)^2$. Combining these two decompositions, we get a good estimate of $K_X^3$ eventually.
\end{rem}

{We recall the following theorem which will play a pivotal role in estimating the lower bound of $K_X^3$:
\begin{thm}[{see Catanese--Chen--Zhang \cite[Theorem~3.2]{CCZ}}] Let $S$ be a smooth projective $(1,2)$-surface and $\sigma_S:S\rightarrow S_0$ the contraction onto its minimal model. Assume that $D\in |\sigma_S^*(K_{S_0})|$ is an effective divisor with simple normal crossing support. Then 
$$h^0(S, K_S+\roundup{\varepsilon D})\geq 3$$
holds for any rational number $\varepsilon>\frac{3}{10}$. 
\end{thm}}

\begin{proof}[Proof of Proposition~\ref{prop: (1,2) pencil}]

Keep the notation in Setting~\ref{setting: pg=5}. We simply write $p_g=p_g(X)$. 
 We may write
\begin{align*}
 \pi^*(K_X)\equiv d S+Z_\pi,
\end{align*}
where $d\geq p_g-1$. 

\medskip

{\bf Step 1. } We make some technical preparations.

Take integers $m\geq 2$ and $n\geq 0$ such that 
\begin{align}
 m-1-\frac{n+1}{p_g-1}>0. \label{eq condition m-1->0}
\end{align} 
Note that 
\begin{align*}
 {}&(m-1)\pi^*(K_X)-(n+1)S-\frac{n+1}{p_g-1}Z_\pi\\
 \equiv {}& \left(m-1-\frac{n+1}{p_g-1}\right)\pi^*(K_X)+\frac{(n+1)(d-p_g+1)}{p_g-1}S
\end{align*}
is nef and big. 
By the Kawamata--Viehweg vanishing theorem,
\[
H^1(W, K_W+\roundup{(m-1)\pi^*(K_X)-(n+1)S-\frac{n+1}{p_g-1}Z_\pi})=0,
\]
which implies that the natural restriction map
\begin{align*}
 {}& H^0(W, K_W+\roundup{(m-1)\pi^*(K_X)-nS-\frac{n+1}{p_g-1}Z_\pi})\\
 \to {}&H^0(S, K_{S}+\roundup{(m-1)\pi^*(K_X)|_{S}-\frac{n+1}{p_g-1}Z_\pi|_{S}}) 
\end{align*}
is surjective. Here in the last term, the restriction on $S$ commutes with the roundup as $S$ is general, and we use the fact that $S|_S\sim 0$. Therefore, 
\begin{align}
 |mK_W-nS||_S\lsgeq |K_S+\roundup{L_{m,n}}+(m-2)\sigma_S^*(K_{S_0})|,\label{eq: K>K+L+K}
\end{align}
 where 
\[L_{m,n}=(m-1)\pi^*(K_X)|_{S}-\frac{n+1}{p_g-1}Z_\pi|_{S}-(m-2)\sigma_S^*(K_{S_0}).\]

 Suppose that 
\begin{align}
 L_{m,n}\geq t\sigma_S^*(K_{S_0}) \text{ for some } t>\frac{3}{10}. \label{eq condition L>3/10}
\end{align}
Here the choice of $\frac{3}{10}$ is to use \cite[Theorem~3.2]{CCZ} where it appears as a key constant. 
Recall that $S$ is a $(1,2)$-surface by assumption, so by \cite[Theorem~3.2]{CCZ}, $h^0(S, K_S+\roundup{L_{m,n}})\geq 3$. 
Denote $|H_{m,n}|=\Mov|K_S+\roundup{L_{m,n}}|$, then by comparing the movable parts of both sides of \eqref{eq: K>K+L+K}, we have
\begin{align}
M_{m, -n}|_S\geq H_{m,n}+(m-2)C, \label{eq: K>H+K} 
\end{align}
where $|M_{m, -n}|= \Mov|\rounddown{\pi^*(mK_X)-nS}|$ and $C\in \Mov|K_S|= \Mov|\sigma_S^*(K_{S_0})|$ is a general element. Here we recall that 
\[\Mov|\rounddown{\pi^*(mK_X)-nS}|=\Mov| mK_W-nS|\] by the isomorphism $H^0(W, mK_W)\simeq H^0(W, \rounddown{\pi^*(mK_X)})$.

Also recall that by \cite[\S3.7]{Chen07}, \begin{align}
  (\pi^*(K_X)\cdot C)=\xi \geq 1.\label{eq: 9.3}
\end{align}
We claim that 
\begin{align}
 (\pi^*(K_X)|_S\cdot H_{m, n}) \geq 
 \frac{2p_g-2}{p_g}, \label{eq KH>2p-2/p}
\end{align}
and
\begin{align}
 (\pi^*(K_X)|_S\cdot H_{m, n}) \geq 2 \text{ if } \pi^*(K_X)|_S-C \text{ is } \mathbb{Q} \text{-effective}. \label{eq KH>2}
\end{align}
If $|H_{m,n}|$ is composed with a pencil, then it is composed with the same pencil as $|K_S|$ and hence $H_{m,n}\geq 2C$ as $h^0(S, H_{m,n})\geq 3$. 
Then by \eqref{eq: 9.3},
$$(\pi^*(K_X)|_S\cdot H_{m, n})\geq (\pi^*(K_X)|_S\cdot 2C)\geq 2.$$ 
If $|H_{m,n}|$ is not composed with a pencil, then 
$(H_{m, n}\cdot C)\geq 2$ as $H_{m,n}|_C$ is movable and $g(C)\geq 2$. 
On the other hand, by Lemma~\ref{lem: 2K>K}, $\pi^*(K_X)|_S-\frac{p_g-1}{p_g}C$ is $\mathbb{Q}$-effective. 
Hence 
\begin{align*}
 (\pi^*(K_X)|_S\cdot H_{m, n})\geq \frac{p_g-1}{p_g}(C\cdot H_{m, n})\geq \frac{2p_g-2}{p_g}.  
\end{align*}
Moreover, if $\pi^*(K_X)|_S- C$ is $\mathbb{Q}$-effective, then the last inequality gives $(\pi^*(K_X)|_S\cdot H_{m, n})\geq 2$.

We remind that to make use of all results in Step 1, it suffices to check conditions \eqref{eq condition m-1->0} and \eqref{eq condition L>3/10}.

\medskip

{\bf Step 2.} We treat the case in Assertion (1). 

Take $n_0=\rounddown{\frac{7p_g-18}{10}}$, then $n_0$ satisfies
\begin{align}\label{eq n0}
 1-\frac{n_0+1}{p_g-1} > \frac{3}{10}.
\end{align}
Here \eqref{eq condition m-1->0} is satisfied for $(m,n)=(2, n_0)$.
In this case, after replacing $X$ by $X_0$, we may assume that $\text{\rm Mov}|K_X|$ is base point free, in other words, $\Phi_{|K_X|}: X\to \Sigma$ is a morphism, and it induces a natural fibration $\phi: X\to \Sigma'$ by the Stein factorization. As $X$ is minimal, the fiber of $\phi$ is exactly $S_0$, so we have $\pi^*(K_X)|_{S}=\sigma_S^*(K_{S_0})$. Hence
\begin{align*}
 L_{2,n_0} \geq \left(1-\frac{n_0+1}{p_g-1}\right) \pi^*(K_X)|_{S} 
 = \left(1-\frac{n_0+1}{p_g-1}\right) \sigma_S^*(K_{S_0}).
\end{align*}
In the last term, the coefficient of $\sigma_S^*(K_{S_0})$ is strictly larger than $\frac{3}{10}$ by \eqref{eq mn}, so \eqref{eq condition L>3/10} is satisfied. So by \eqref{eq: K>H+K}, $M_{2,-n_0}|_S\geq H_{2, n_0}$.
As $\pi^*(K_X)|_{S}=\sigma_S^*(K_{S_0})\geq C$, we have $(\pi^*(K_X)|_S\cdot H_{2, n_0})\geq 2$ by \eqref{eq KH>2}.
By definition, $\pi^*(2K_X)\geq n_0S+M_{2,-n_0}$.
So
\begin{align*}
 2K_X^3\geq {}& (\pi^*(K_X)^2\cdot (n_0S+M_{2,-n_0})) \\
 = {}& n_0(\sigma_S^*(K_{S_0}))^2+(\pi^*(K_X)^2\cdot M_{2,-n_0}) \\
 \geq {}& n_0 + (p_g-1)(S\cdot \pi^*(K_X) \cdot M_{2,-n_0})\\
 \geq {}&n_0 +(p_g-1)(\pi^*(K_X)|_S \cdot H_{2,n_0})\\
 \geq {}&n_0+2(p_g-1)=\rounddown{\frac{27p_g-38}{10}}.
\end{align*}

\medskip

{\bf Step 3.} We treat the case in Assertion (2). 

In this case, $|K_X|$ is composed with a rational pencil by \cite[Corollary~3.3(1)]{Noether}.

Take $m\geq 2$ and $n\geq 0$ satisfying
\begin{align}\label{eq mn}
 \left(m-1-\frac{n+1}{p_g-1}\right)\cdot \frac{p_g-1}{p_g}>m-2+\frac{3}{10},
\end{align}
or equivalently, 
\[
m+n<\frac{7p_g}{10}.
\]
Here \eqref{eq condition m-1->0} is satisfied. 
Recall that by our construction, $\pi$ factorizes through $\mu$ in Setting~\ref{set: Special modification}. Hence by Lemma~\ref{lem: 2K>K}, 
\begin{align*}
 L_{m,n} \geq {}& \left(m-1-\frac{n+1}{p_g-1}\right) \pi^*(K_X)|_{S}-(m-2)\sigma_S^*(K_{S_0})\\
 \geq {}& \left(m-1-\frac{n+1}{p_g-1}\right)\frac{p_g-1}{p_g}\sigma_S^*(K_{S_0})-(m-2)\sigma_S^*(K_{S_0})\\
 ={}&\left(\left(m-1-\frac{n+1}{p_g-1}\right)\frac{p_g-1}{p_g}-(m-2)\right)\sigma_S^*(K_{S_0}). 
\end{align*}
In the last term, the coefficient of $\sigma_S^*(K_{S_0})$ is strictly larger than $\frac{3}{10}$ by \eqref{eq mn}, so \eqref{eq condition L>3/10} is satisfied.

Then we consider special values of $(m,n)$. Take $m_1=\rounddown{\frac{7p_g-1}{10}}$ and $n_1=\rounddown{\frac{7p_g-21}{10}}=m_1-2$.

Note that $(m, n)=(m_1, 0)$ satisfies \eqref{eq mn}. 
Then by \eqref{eq: K>H+K}, 
\[\pi^*(m_1K_X)|_S\geq M_{m_1, 0}|_S\geq H_{m_1,0}+(m_1-2)C.\]
By
\eqref{eq: 9.3} and \eqref{eq KH>2p-2/p},
\begin{align}
m_1(\pi^*(K_X)|_S)^2\geq{}& (\pi^*(K_X)|_S\cdot H_{m_1, 0})+(m_1-2)(\pi^*(K_X)|_S\cdot C)\notag\\
\geq{}&\frac{2p_g-2}{p_g}+m_1-2.\label{eq L2>=}
\end{align}

On the other hand, $(m,n)=(2,n_1)$ satisfies \eqref{eq mn}.
So by \eqref{eq: K>H+K}, $M_{2,-n_1}|_S\geq H_{2, n_1}$.
By definition, $\pi^*(2K_X)\geq n_1S+M_{2,-n_1}$.
So
\begin{align*}
 2K_X^3\geq {}& (\pi^*(K_X)^2\cdot (n_1S+M_{2,-n_1})) \\
 = {}& n_1(\pi^*(K_X)|_S)^2+(\pi^*(K_X)^2\cdot M_{2,-n_1}) \\
 \geq {}& n_1(\pi^*(K_X)|_S)^2+ (p_g-1)(S\cdot \pi^*(K_X) \cdot M_{2,-n_1})\\
 \geq {}&n_1(\pi^*(K_X)|_S)^2+(p_g-1)(\pi^*(K_X)|_S \cdot H_{2,n_1})\\
 \geq {}&\frac{n_1}{m_1}\left(\frac{2p_g-2}{p_g}+m_1-2\right)+\frac{2(p_g-1)^2}{p_g}.
\end{align*}
Here for the last step we use \eqref{eq KH>2p-2/p} and \eqref{eq L2>=}.
The proof is completed. 
 \end{proof}
By Proposition~\ref{prop: (1,2) pencil}, we have the following theorem.
 \begin{thm} \label{thm: noether 12 slop} Let $X$ be a minimal $3$-fold of general type with $p_g(X)\geq 3$. Suppose that $|K_X|$ is composed with a pencil of $(1,2)$-surfaces, then 
\[K_X^3\ge \frac{1}{2}\rounddown{\frac{27p_g(X)-38}{10}},\]
with the following exceptional cases:
\renewcommand{\arraystretch}{1.5}
\[\begin{tabular}{|r|c|c|c|c|c|c|c|c|c|}
		\hline
		 	$p_g(X)=$ & $3$ & $4$& $5$ & $6$ & $7$& $8$& $9$ & $10$\\
 \hline
 $K_X^3\geq $ & $\frac{4}{3}$ & $\frac{9}{4}$ & $\frac{109}{30}$& $\frac{61}{12}$ & $\frac{85}{14}$& $\frac{151}{20}$ &$\frac{244}{27}$ & $\frac{301}{30}$\\
		\hline
	\end{tabular}\]
\end{thm}

\begin{proof} If we are in the situation of Proposition~\ref{prop: (1,2) pencil}(1), then the inequality is proved. 
 If we are in the situation of Proposition~\ref{prop: (1,2) pencil}(2), then $p_g(X)\leq 10$ by \cite[Corollary~3]{Noether_Add} and hence the inequality can be checked case by case directly.
\end{proof}

By Propositions~\ref{prop: not (1,2)} 
and~\ref{prop: (1,2) pencil}, we have the following theorem.

\begin{thm}\label{thm: volume pg=5}
 Let $X$ be a minimal $3$-fold of general type with $p_g(X)\geq 5$.  Then 
\[K_X^3\ge \frac{1}{4}\roundup{\frac{8}{3}(2p_g(X)-5)}.\]
In particular, when $5\leq p_g(X)\leq 10$, the lower bound of $K_X^3$ is given by the following table: 
\renewcommand{\arraystretch}{1.5}\[\begin{tabular}{|r|c|c|c|c|c|c|c|}
		\hline
		 	$p_g(X)=$ & $5$ & $6$ & $7$& $8$& $9$ & $10$\\
 \hline
 $K_X^3\geq $ & $\frac{7}{2}$& $\frac{19}{4}$ & $6$& $\frac{15}{2}$ &$\frac{35}{4}$ & $10$\\ 
		\hline
	\end{tabular}\]
\end{thm}

\begin{proof}
By {\cite[Theorem~1.4, Theorem~1.5]{Chen07}}, we may assume that $p_g(X)\geq 5$.

Suppose that $|K_X|$ is composed with a pencil of $(1,2)$-surfaces, then by Theorem~\ref{thm: noether 12 slop}, 
$K_X^3>\frac{1}{4}\roundup{\frac{8}{3}(2p_g(X)-5)}$. In fact,  
it is easy to check that $K_X^3\ge \frac{1}{2}\rounddown{\frac{27p_g-38}{10}}>\frac{1}{4}\roundup{\frac{8}{3}(2p_g(X)-5)}$ when $p_g(X)\geq 11$ as the slope is larger;  
 the exceptional cases can be checked case by case directly.

Suppose that $|K_X|$ is not composed with a pencil of $(1,2)$-surfaces, then we get the conclusion by Proposition~\ref{prop: not (1,2)}.
\end{proof}
\begin{proof}[Proof of Theorem~\ref{thm: pg 5}]
It follows directly from Theorem~\ref{thm: volume pg=5}.
\end{proof}
\begin{rem}
\begin{enumerate}
 \item When $p_g(X)=5$, the lower bound in Theorem~\ref{thm: volume pg=5} is sharp by {the example} \cite[Table~10, No.~2]{CJL} and this example justifies the optimality of Proposition~\ref{prop: not (1,2)}; when $p_g(X)=7$ or $10$, the lower bounds in Theorem~\ref{thm: volume pg=5} are sharp by examples in \cite[Proposition~3.2]{Kob}; when $p_g(X)=8$, the lower bound in Theorem~\ref{thm: volume pg=5} is sharp by an example constructed in the same way as \cite[Proposition~4.15(2)]{HZ}. 

 \item More surprisingly, the lower bounds $\frac{9}{4}, \frac{109}{30}, \frac{85}{14}, \frac{151}{20}, \frac{301}{30}$ in Theorem~\ref{thm: noether 12 slop} 
 are sharp when $p_g(X)=4, 5,7,8,10$.
Such values are attained by examples \cite[Table~10, No.~3, No.~8--10]{CJL} and \cite[Section~6, Table~1]{CP23}. When $p_g(X)=6$, the lower bound $\frac{61}{12}$ in Theorem \ref{thm: noether 12 slop} is also optimal by an example constructed by S. Coughlan, Y. Hu, R. Pignatelli and T. Zhang in a paper in preparation. 

 \item By \cite[Proposition~4.5]{HZ}, the lower bound of Theorem~\ref{thm: volume pg=5} can be improved to $9$ when $p_g(X)=9$; and the lower bound can be to $5$ when $p_g(X)=6$ by the same method, but we will not write down the details here. 
\end{enumerate}
 
\end{rem}
We would like to raise the following question.
\begin{question}
 Is the lower bound $\frac{244}{27}$ optimal {among all those} $3$-folds $X$ canonically fibered by $(1,2)$-surfaces with $p_g(X)=9$?
\end{question}

\section{Birational geometry of $3$-folds with $p_g=2$ and $\Vol=1/3$}\label{sec: pg2}

In this section, we study the birational geometry of a minimal $3$-fold $X$ of general type with $p_g(X)=2$ and $K_X^3=\frac{1}{3}.$ 

We briefly explain our strategy.
The final goal is to determine all the plurigenera and the behavior of pluricanonical systems of $X$. We have well-developed methods to get the latter from the former (see for example \cite{EXPIII, CHP21}). In order to get all the plurigenera, according to Reid's Riemann--Roch formula, we need to know the singularities of $X$ explicitly. The local information of singularities can be read out by the strong global restriction that 
$p_g(X)=2$ and $K_X^3=\frac{1}{3}$ by the following: first, we study the geometric structure of $X$ in detail and control the Gorenstein indices of singularities of $X$ in \S\ref{sec: 4.1}; then we prove inequalities of $P_2(X)$ and $P_3(X)$ in \S\ref{sec 4.2}; combining these with Reid's Riemann--Roch formula will lead us to the clear description of singularities of $X$.

\subsection{Geometric structure and non-Gorenstein singularities}\label{sec: 4.1}\

Since $p_g(X)=2$, $|K_X|$ is composed with a pencil. Take a birational modification $\mu: X'\to X$ as in Setting~\ref{set: Special modification} and keep the same notation. Since $h^1(\mathcal{O}_X)=0$ by \cite[Theorem~1.5(3)]{Chen07}, we have $\Gamma\simeq \mathbb{P}^1$, in other words, $|K_X|$ is composed with a rational pencil. By \cite[Theorem~1.4(2)]{Chen07}, $F$ is a $(1,2)$-surface. 

We recall the following lemma from \cite[\S2.15]{Chen07}.
\begin{lem}\label{lem: canonical deg} 
Let $X$ be a minimal $3$-fold of general type with $p_g(X)=2$ and $K_X^3=\frac{1}{3}$. Keep the notation in Setting~\ref{set: Special modification}. Then 
 \[
 (\mu^*(K_X)|_F)^2=\frac{1}{3},\quad (\mu^*(K_X)|_{F}\cdot \sigma^*(K_{F_0}))=(\mu^*(K_X)|_F\cdot \Mov|\sigma^*(K_{F_0})|)=\frac{2}{3}.
 \]
\end{lem}
\begin{proof}
 For the convenience of readers who are not familiar with \cite{Chen07}, we briefly recall the proof. 
Denote by $\xi=(\mu^*(K_X)|_F\cdot \Mov|\sigma^*(K_{F_0})|)$ (which coincides with the definition in \cite[\S2.9]{Chen07}). Then by the first $3$ lines of \cite[\S2.15]{Chen07}, one has $\xi\geq\frac{2}{3}$.
By Lemma~\ref{lem: 2K>K}, $2\mu^*(K_X)|_{F}-\sigma^*(K_{F_0})$ is effective. Then 
\begin{align*}
 K_X^3={}&\mu^*(K_X)^3\geq (\mu^*(K_X)^2\cdot F)=(\mu^*(K_X)|_F)^2\notag \\\geq {}&\frac{1}{2}(\mu^*(K_X)|_F\cdot \sigma^*(K_{F_0}))\geq \frac{\xi}{2}\geq \frac{1}{3}. 
\end{align*}
Then by assumption, all the above inequalities are equalities. 
\end{proof}

Now we will analyze the geometry of $X$ in more detail in the next $2$ lemmas.

\begin{lem}\label{lem: excep1}
Let $X$ be a minimal $3$-fold of general type with $p_g(X)=2$ and $K_X^3=\frac{1}{3}$. Keep the notation in Setting~\ref{set: Special modification}.
Then there exists a unique $\mu$-exceptional prime divisor $E_0$ such that $(\sigma^*(K_{F_0})\cdot E_0|_{F})>0$. Moreover, 
 \begin{enumerate}
 \item $(\sigma^*(K_{F_0})\cdot E_0|_{F})=1$;
 \item $\coeff_{E_0}E_\mu=\frac{1}{3}$, $\coeff_{E_0}Z=\frac{2}{3}$;
 \item $\mu(E_0)$ is a point.
 \end{enumerate}
\end{lem}
\begin{proof}
 By the adjunction formula, we have
 \[
 K_{F}=\mu^*(K_X)|_{F}+E_{\mu}|_{F}.
 \]
 By Lemma~\ref{lem: canonical deg} 
 and the equality $(K_{F}\cdot\sigma^*(K_{F_0}))=K_{F_0}^2=1$, we have 
 \begin{align}
 (E_{\mu}|_{F}\cdot\sigma^*(K_{F_0}))=\frac{1}{3}. \label{eq: Ealpha.K=1/3}
 \end{align}
 Thus there exists a $\mu$-exceptional prime divisor $E_0\subseteq \Supp E_{\mu}$ such that $(E_{0}|_{F}\cdot\sigma^*(K_{F_0}))>0,$ which means that $(E_{0}|_{F}\cdot\sigma^*(K_{F_0}))\geq 1.$ Hence 
 \begin{align}((K_{X'}-E_{0})|_{F}\cdot\sigma^*(K_{F_0}))=(K_{F}\cdot\sigma^*(K_{F_0}))-(E_{0}|_{F}\cdot\sigma^*(K_{F_0}))\leq 0.\label{eq: K-E<=0}
 \end{align}
On the other hand, 
\[K_{X'}=F+Z+E_\mu\]
is an effective integral divisor whose support contains all $\mu$-exceptional divisors. So $K_{X'}-E_0$ is an effective integral divisor.
As $\sigma^*(K_{F_0})$ is nef, we conclude from \eqref{eq: K-E<=0} that 
\begin{align}
 {}&((K_{X'}-E_{0})|_{F}\cdot\sigma^*(K_{F_0})) =0;\label{eq: K-E.K=0}\\
 {}&(E_{0}|_{F}\cdot\sigma^*(K_{F_0}))= 1.\label{eq: E0.K=1}
\end{align}
 In particular, \eqref{eq: K-E.K=0} implies that $\coeff_{E_0}(Z+E_\mu)=1$ and $E_0$ is the 
 unique $\mu$-exceptional prime divisor such that $(E_{0}|_{F}\cdot\sigma^*(K_{F_0}))>0$.
 Furthermore, \eqref{eq: Ealpha.K=1/3} and \eqref{eq: E0.K=1} imply that $\coeff_{E_0}E_\mu=\frac{1}{3}$ and hence $\coeff_{E_0}Z=\frac{2}{3}$.

 Finally, since $X$ has only isolated singularities and $\coeff_{E_0}E_\mu\not \in \mathbb{Z}$, we deduce that $\mu(E_0)$ is a point. 
\end{proof}

\begin{lem}\label{lem: excep2} Let $X$ be a minimal $3$-fold of general type with $p_g(X)=2$ and $K_X^3=\frac{1}{3}$. Keep the notation in Setting~\ref{set: Special modification} and  
Lemma~\ref{lem: excep1}. Then 
 \begin{enumerate}
 \item $E_0|_{F}=C_0$ is an integral $(-3)$-curve;
 
 \item $G|_F =\frac{1}{3}C_0$;
 \item $ Z_X=0$.
 \end{enumerate}
\end{lem}
\begin{proof}
 By Lemma~\ref{lem: excep1}(1), there is a unique integral curve $C_0$ on $F$ with $\coeff_{C_0}E_0|_{F}=1$ such that $(\sigma^*(K_{F_0})\cdot C_0)=1$. In particular, this implies that $C_0$ is not contracted by $\sigma$ and hence
 \begin{align}
 (K_F\cdot C_0)\ge (\sigma^*(K_{F_0})\cdot C_0)=1.\label{eq: KC>1}
 \end{align} 
Note that in Setting~\ref{set: Special modification}, 
\begin{align}\label{eq: Z-E=Z+G}Z-E_{\mu}=\mu^*(2K_X)-K_{X'}-F=\mu^*(Z_X)+G
\end{align}
is an effective $\mathbb{Q}$-divisor. Hence 
\begin{align}\label{eq: Z-E=2K-F}(Z-E_{\mu})|_F=\mu^*(2K_X)|_F-K_{F}
\end{align}is an effective $\mathbb{Q}$-divisor with \[\coeff_{C_0}(Z-E_{\mu})|_F=\coeff_{E_0}(Z-E_{\mu})=\frac{1}{3}\] by Lemma~\ref{lem: excep1}(2). 
Hence by \eqref{eq: Z-E=2K-F},
\begin{align}\label{eq: C<-KC}
 \frac{1}{3}C_0^2\leq \big((Z-E_{\mu})|_F\cdot C_0\big)=-(K_F\cdot C_0).
\end{align}
Here we used the fact that $C_0$ is contracted by $\mu$
by Lemma~\ref{lem: excep1}(3).  
Combining \eqref{eq: KC>1} and \eqref{eq: C<-KC}, we have
\[
(K_F\cdot C_0)+C_0^2\leq -2(K_F\cdot C_0)\leq -2,
\]
 According to the genus formula, we deduce that $
 (K_{F}\cdot C_0)+C_0^2=-2$ and $C_0^2=-3$, which means that $C_0$ is a $(-3)$-curve.
 
Since $\sigma^*(K_{F_0})$ is nef and 
\[
\left(\sigma^*(K_{F_0})+\frac{1}{3}C_0\right)\cdot C_0=1-1=0,
\]
we know that $\sigma^*(K_{F_0})+\frac{1}{3}C_0$ is nef. Moreover,  
$(\sigma^*(K_{F_0})+\frac{1}{3}C_0)^2=\frac{4}{3}.$
On the other hand,
from \eqref{eq: Z-E=2K-F} we have
\begin{align}
 \mu^*(2K_X)|_F={}& K_{F}+(Z-E_{\mu})|_F\geq K_{F}+\frac{1}{3}E_0|_F \notag\\
 \geq{}& K_{F}+\frac{1}{3}C_0\geq \sigma^*(K_{F_0})+\frac{1}{3}C_0. \label{eq: long zariski}
\end{align}
Here $(\mu^*(2K_X)|_F)^2=\frac{4}{3}$ 
 by Lemma~\ref{lem: canonical deg}. Hence by Lemma~\ref{lem: HIT}(1), all inequalities in \eqref{eq: long zariski} become equalities. 
 In particular,  
 $(Z-E_{\mu})|_F=\frac{1}{3}E_0|_F=\frac{1}{3}C_0$.  
 Moreover,
 $(\mu^*(Z_X)+G)|_F=\frac{1}{3}C_0$ by \eqref{eq: Z-E=Z+G}, so $\mu^*(Z_X)|_F$ and $G|_F$ are multiples of $C_0$. Since $C_0$ is contracted by $\mu$
by Lemma~\ref{lem: excep1}(3), we have $\mu^*(Z_X)|_F=0$ and $G|_F=\frac{1}{3}C_0$. The last assertion follows from Lemma~\ref{lem: Z=0}. 
\end{proof}

Now we can study the Gorenstein indices of singularities of $X$ by Lemma~\ref{lem: plt}.
\begin{lem}\label{lem: singularty of X}
Let $X$ be a minimal $3$-fold of general type with $p_g(X)=2$ and $K_X^3=\frac{1}{3}$. 
 Suppose that $P$ is a non-Gorenstein singularity of $X$. 
Then $r_P\in \{2, 3\}$ and there is exactly one such $P$ with $r_P=3$. Here $r_P$ is the Cartier index of $K_X$ at $P$.
\end{lem}
\begin{proof}
By Lemma~\ref{lem: excep2}, all conditions of Lemma~\ref{lem: plt} are satisfied.

We claim that $F_X$ has exactly one non-Gorenstein singularity which is of type $\frac{1}{3}(1,1)$. Indeed, by Lemmas~\ref{lem: excep2} and \ref{lem: plt}, 
\[
 K_F+\frac{1}{3}C_0 =(\mu|_{F})^*(K_{F_X}).
\]
So outside the point $P_0=\mu|_{F}(C_0)$, $F_X$ has at worst Du Val singularities which are clearly Gorenstein. The fiber of $\mu|_{F}$ over $P_0$ has simple normal crossing support as $F_X$ has only rational singularities, so the fiber of $\mu|_{F}$ over $P_0$ is exactly $C_0$, otherwise there exists an exceptional curve $C'$ connecting $C_0$ with \[(K_F\cdot C')=\frac{1}{3}(-C_0\cdot C')=-\frac{1}{3},\] which is absurd. Since $C_0$ is a $(-3)$-curve, it is clear that $C_0$ is contracted to a cyclic quotient singularity of type $\frac{1}{3}(1,1)$. 
 
Now to conclude the lemma, by Lemma~\ref{lem: plt}(2), we only need to consider the case that $P=P_0$ is the unique non-Gorenstein singularity of $F_X$.
 Then $r_P=3$ as 
$3K_X|_{F_X}$ is Cartier at $P$. 
\end{proof}

\subsection{Geometry of pluricanonical systems}\label{sec 4.2}\

In this subsection, we study the pluricanonical maps of $X$.

Let $\pi: W\to X$ be as in Setting~\ref{Set up for canonical and pluricanonical maps} and keep the same notation. We may further assume that $\pi$ factorizes through the modification $\mu: X'\to X$ in \S\ref{sec: 4.1}. In this case, $S$ is a $(1,2)$-surface. Denote by $\sigma_S:S\to S_0=F_0$ the morphism to its minimal model. Let $C$ be a general member of $\Mov |K_S|=\Mov |\sigma_S^*K_{S_0}|$. Then $C$ is a smooth projective curve of genus $2$. Here for a good reference on geometry of $(1,2)$-surfaces, we refer to \cite[\S2]{Horikawa}.

Recall that 
\begin{align}
(\pi^*(K_X)|_S)^2{}&=\frac{1}{3},\label{eq: 4.2-1}\\
 (\pi^*(K_X)|_S\cdot C){}&=
 \frac{2}{3}\label{eq: 4.2-2}
\end{align}
by Lemma~\ref{lem: canonical deg} and the projection formula.

\begin{lem}\label{lem: bound of P2}
Let $X$ be a minimal $3$-fold of general type with $p_g(X)=2$ and $K_X^3=\frac{1}{3}$. Then $P_2(X)\le 4$. 
\end{lem}
\begin{proof}
Keep the notation in Setting~\ref{Set up for canonical and pluricanonical maps}.
Recall that $\pi: W\to X$ is a resolution, $S$ is a general irreducible element of $\Mov |\rounddown{\pi^*(K_X)}|$. For any integers $0<m\le 4$ and $0\leq j\leq m$, we consider the linear system \[|M_{m, -j}|:= \Mov|\rounddown{\pi^*(mK_X)-jS}|\] with restriction map
\[
\theta_{m,-j}: H^0(W, M_{m,-j})\to H^0(S, M_{m,-j}|_S),
\]
and set $u_{m,-j}=\dim \Image \theta_{m,-j}$.

\medskip 

 \textbf{Step 1.} We claim that $u_{2,0}\le 2$.

 Suppose that $u_{2,0}\ge 3$. If $|M_{2,0}||_S$ and $|C|$ are composed with the same pencil, then $M_{2,0}|_S\ge 2C$ since $u_{2,0}\ge 3$. Thus $2\pi^*(K_X)|_S\ge 2C$, which implies that 
 \[
 K_X^3\geq (\pi^*(K_X)|_S)^2 \ge (\pi^*(K_X)|_S\cdot C)= \frac{2}{3},
\]
 which is a contradiction. On the other hand, if $|M_{2,0}||_S$ and $|C|$ are not composed with the same pencil, then $(M_{2,0}|_S\cdot C)\ge 2$ as $|M_{2,0}|$ induces a nontrivial movable linear system on $C$. Hence
 \[
 (\pi^*(K_X)|_S\cdot C)\ge\frac{1}{2}(M_{2,0}|_S\cdot C)\ge 1,
\]
which contradicts \eqref{eq: 4.2-2}. We conclude that $u_{2,0}\le 2$.

\medskip

 \textbf{Step 2.} We claim that $u_{2,-1}\le 1$. 

 Suppose that $u_{2,-1}\ge 2$.
 If $|M_{2,-1}||_S$ and $|C|$ are composed with the same pencil, then $M_{2,-1}|_S\geq C$ and hence
 \begin{align*}
 2K_X^3 &\ge (\pi^*(K_X)^2\cdot S)+(\pi^*(K_X)^2\cdot M_{2,-1}) \\
 &\ge(\pi^*(K_X)^2\cdot S)+(\pi^*(K_X)|_S\cdot M_{2,-1}|_S)\\
 &\ge(\pi^*(K_X)^2\cdot S)+(\pi^*(K_X)|_S\cdot C)\\
 &\ge 1,
 \end{align*}
 which is a contradiction. 
On the other hand, if $|M_{2,-1}||_S$ and $|C|$ are not composed with the same pencil, then $(M_{2,-1}|_S\cdot C)\geq 2$. 
Hence
 \[
 (\pi^*(K_X)|_S\cdot C)\ge\frac{1}{2}(M_{2,-1}|_S\cdot C)\ge 1,
\]
which contradicts \eqref{eq: 4.2-2}. Thus $u_{2,-1}\le 1$.

\medskip

\textbf{Step 3.} We claim that $h^0(W, M_{2, -2})\le 1$.

Otherwise, 
$|M_{2,-2}|$ is a nontrivial movable linear system. Since $K_X$ is nef and big, we have $(\pi^*(K_X)^2\cdot M_{2,-2})>0$. 
Recall that $2\pi^*(K_X)-2S\geq M_{2,-2}$ by definition. Then by \eqref{eq: 4.2-1}, 
\[
K_X^3\ge (\pi^*(K_X)|_S)^2+\frac{1}{2}((\pi^*(K_X))^2\cdot M_{2,-2})>\frac{1}{3},
\]
which is a contradiction. Thus $h^0(W, M_{2, -2})\le 1$.

\medskip

\textbf{Step 4.} The conclusion that $P_2(X)\le 4$.

This follows from 
\[
P_2(X)=h^0(W, \rounddown{2\pi^*(K_X)})= u_{2,0}+u_{2,-1}+h^0(W, M_{2, -2}) 
\]
and previous steps. 
\end{proof}

\begin{lem}\label{lem: bound of P3 and P4}
Let $X$ be a minimal $3$-fold of general type with $p_g(X)=2$ and $K_X^3=\frac{1}{3}$. 
Then
\begin{enumerate}
 \item $P_3(X)\ge P_2(X)+3$;
 \item $|3K_X|$ defines a generically finite map;
 \item $P_4(X)\ge P_3(X)+4$.
\end{enumerate}
\end{lem}
\begin{proof}
Keep the notation in Setting~\ref{Set up for canonical and pluricanonical maps}.

By the Kawamata--Viehweg vanishing theorem, the natural restriction map
\[
H^0(W, K_{W}+\roundup{\pi^*(K_X)}+S)\to H^0(S, K_S+\roundup{\pi^*(K_X)|_S})
\]
is surjective. Here in the last term, the restriction on $S$ commutes with the roundup as $S$ is general. So
\begin{align}
|3K_{W}||_S\lsgeq |K_{W}+\roundup{\pi^*(K_X)}+S||_S\lsgeq |K_S+\roundup{\pi^*(K_X)|_S}|.\label{eq: 3K>KS+K}
\end{align}
Note that $2\pi^*(K_X)|_S\ge \sigma_S^*(K_{S_0})$ by Lemma~\ref{lem: 2K>K} as $\pi$ factorizes through $\mu$. 
By \cite[Theorem~3.2]{CCZ}, we have $h^0(S, K_S+\roundup{\pi^*(K_X)|_S})\ge 3$. Then $u_{3,0}\ge 3$, which implies that 
\[P_3(X)=h^0(W, \rounddown{3\pi^*(K_X)}-S)+u_{3,0}\ge P_2(X)+3.\] 
Assertion (1) is proved.

Next, we claim that $|K_S+\roundup{\pi^*(K_X)|_S}|$ defines a generically finite map on $S$. Otherwise, it is composed with the same pencil with $|C|$ and thus $K_S+\roundup{\pi^*(K_X)|_S}\ge 2C$ since $h^0(S, K_S+\roundup{\pi^*(K_X)|_S})\ge 3$. By comparing the movable parts of both sides of \eqref{eq: 3K>KS+K}, we know that $3\pi^*(K_X)|_S\geq 2C$. Thus by \eqref{eq: 4.2-2}, 
\[
K_X^3\ge (\pi^*(K_X)|_S)^2\ge \frac{2}{3}(\pi^*(K_X)|_S\cdot C)= \frac{4}{9}, 
\]
which is a contradiction. So $|K_S+\roundup{\pi^*(K_X)|_S}|$ defines a generically finite map on $S$. 
By \eqref{eq: 3K>KS+K}, $|3K_X|$ defines a generically finite map. 
Assertion (2) is proved.

By \cite[Theorem~2.4(1)]{CJ}, the natural restriction map
\[
H^0(W, 2(K_{W}+S))\to H^0(S,2K_S)
\]
is surjective. Note that $P_2(S)=4$ and $|4K_{W}|\lsgeq |2(K_{W}+S)|$. We deduce that $u_{4,0}\ge 4$, which implies that $P_4(X)\ge P_3(X)+4$.
\end{proof}

\begin{prop}\label{prop: RR X16}
Let $X$ be a minimal $3$-fold of general type with $p_g(X)=2$ and $K_X^3=\frac{1}{3}$. Then
\begin{enumerate}
\item $\chi(\mathcal{O}_X)=-1$, $P_2(X)=4$;
 \item the Reid basket of $X$ is $\{2\times (1, 2), (1,3)\}$;
 \item the Hilbert series of $X$ is given by
 \[
\sum_{k\geq 0}h^0(X, kK_X)q^k=\frac{1-q^{16}}{(1-q)^2 (1-q^{2})(1-q^{3})(1-q^{8})}.
\]
\end{enumerate}
All the above data are the same as those of a general 
hypersurface of degree $16$ in $\mathbb{P}(1,1,2,3,8)$.
\end{prop}
\begin{proof}
By Lemma~\ref{lem: singularty of X} and \cite[(6.1), (6.4)]{Rei87}, the basket $B_X$ of $X$ is $\{a\times (1,2), b\times (1,3)\}$, for some integers $a\geq 0$ and $b\ge 1$. 
By Reid's Riemann--Roch formula \eqref{eq: RR} and the Kawamata--Viehweg vanishing theorem, 
\begin{align*}
 P_2(X){}&=\frac{1}{6} -3\chi(\mathcal{O}_X)+\frac{a}{4}+\frac{b}{3};\\
 P_3(X){}&=\frac{5}{6} -5\chi(\mathcal{O}_X)+\frac{a}{4}+\frac{2b}{3};\\
 P_4(X){}&=\frac{7}{3} -7\chi(\mathcal{O}_X)+\frac{a}{2}+\frac{2b}{3}.
\end{align*}
Then it is clear that
\[
2P_4(X)-P_3(X)-3P_2(X)=\frac{10}{3}-\frac{b}{3}\leq 3. 
\]
On the other hand, by Lemmas~\ref{lem: bound of P3 and P4} and \ref{lem: bound of P2},
\[
2P_4(X)-P_3(X)-3P_2(X)=P_4(X)-P_3(X)+P_4(X)-P_2(X)-2P_2(X)\geq 3. 
\]
Then all the above inequalities become equalities, which means that $b=1$, $P_2(X)=4$, $P_3(X)=7$, $P_4(X)=11$. 
Then we can solve the equations and get $\chi(\mathcal{O}_X)=-1$ and $a=2$.
Assertions (1) and (2) are proved.

Note that Reid's Riemann--Roch formula only depends on $K_X^3$, $\chi(\mathcal{O}_X)$ and $B_X$. 
These data are the same as those of a general hypersurface of degree $16$ in $\mathbb{P}(1,1,2,3,8)$ (see \cite[Table~3, No.~12]{Fle00}). 
So the Hilbert series of them also coincide and are certainly computable (see \cite[Theorem~3.4.4]{WPS}).
\end{proof}

\begin{prop}\label{prop: |mK|}
Let $X$ be a minimal $3$-fold of general type with $p_g(X)=2$ and $K_X^3=\frac{1}{3}$. Then
\begin{enumerate}
 \item $|2K_X|$ is not composed with a pencil;
 \item $|3K_X|$ defines a generically finite map;
 \item $|8K_X|$ defines a birational map.
\end{enumerate}
\end{prop}
\begin{proof}
If $|2K_X|$ is composed with a pencil, then it is the same pencil as $|K_X|$ and $2\pi^*(K_X)\ge 3S$ since $P_2(X)=4$. Thus by \eqref{eq: 4.2-1},
\[
K_X^3\ge \frac{3}{2}(\pi^*(K_X)|_S)^2\ge \frac{1}{2},
\]
which is a contradiction. This proves Assertion (1).


Assertion (2) follows from Lemma~\ref{lem: bound of P3 and P4}.
Assertion (3) follows from \cite[Theorem~1.4]{Chen07}.
\end{proof}

\section{Birational geometry of $3$-folds with $p_g=3$ and $\Vol=1$}\label{sec: pg3}

In this section, we study the birational geometry of a minimal $3$-fold $X$ of general type with $p_g(X)=3$ and $K_X^3=1$. 

\subsection{Singularities}\label{sec: 5.1}\

In this subsection, we study the Gorenstein index of $X$.

Take a birational modification $\mu: X'\to X$ as in Setting~\ref{set: Special modification} and keep the same notation. By \cite[Theorem~1.5(4)]{Chen07}, the canonical image of $X$ is $\mathbb{P}^2$. By \cite[Lemma~5.2]{CHP21}, $F$ is a $(4,4)$-surface. 

The geometry of $X$ is simpler than that in \S\ref{sec: pg2}.
\begin{lem}\label{lem: canonical rest pg3}
 Let $X$ be a minimal $3$-fold of general type with $p_g(X)=3$ and $K_X^3=1$. Keep the notation in Setting~\ref{set: Special modification}. Then 
 \begin{enumerate}
 \item $\mu^*(K_X)|_F\equiv \frac{1}{2} K_F\equiv \frac{1}{2}\sigma^*(K_{F_0})$;
 in particular, $F$ is minimal;
 \item $G|_F=0$, $Z_X=0$.
 \end{enumerate}
 
\end{lem}
\begin{proof}
 By \eqref{eq: K+F>K+F}, we have
\[
 K_F+G|_F=\mu^*(K_X+F_X)|_F.
\]
 Since $G$ is effective, we deduce that
\begin{align}
 2\mu^*(K_X)|_F=K_F+G|_F+\mu^*(Z_X)|_F\ge K_F\ge \sigma^*(K_{F_0}).\label{eq: zariski 2} 
\end{align}
Then $(\mu^*(K_X)|_F)^2\geq \frac{1}{4}K_{F_0}^2=1$.
On the other hand,
$(\mu^*(K_X)|_F)^2\leq K_X^3=1$. 
Hence $(\mu^*(K_X)|_F)^2=\frac{1}{4}K_{F_0}^2=1$.
 By Lemma~\ref{lem: HIT}(1), inequalities in \eqref{eq: zariski 2} become equalities.
 In particular, Assertion (1) holds and 
$G|_F=\mu^*(Z_X)|_F=0$. 
The last assertion follows from Lemma~\ref{lem: Z=0}. 
\end{proof}

\begin{lem}\label{lem: singularity of X pg3}
Let $X$ be a minimal $3$-fold of general type with $p_g(X)=3$ and $K_X^3=1$. 
 Suppose that $P$ is a non-Gorenstein singularity of $X$. Then $r_P=2$. Here $r_P$ is the Cartier index of $K_X$ at $P$.
\end{lem}

\begin{proof}
 By Lemma~\ref{lem: canonical rest pg3}, all conditions of Lemma~\ref{lem: plt} are satisfied. Moreover, $G|_F=0$ implies that $K_F=(\mu|_F)^*(K_{F_X})$. So $F_X$ has at worst Du Val singularities which are clearly Gorenstein, which implies that $r_P=2$ by Lemma~\ref{lem: plt}(2).
\end{proof}

\subsection{Geometry of pluricanonical systems}\

Take a birational modification $\pi: W\to X$ as in Setting~\ref{Set up for canonical and pluricanonical maps} and keep the same notation.
We may further assume that $\pi$ factorizes through the modification $\mu: X'\to X$ in \S\ref{sec: 5.1}. 
In this case, $S$ is a $(4,4)$-surface. By \cite[\S3.2]{Chen07}, $|S|_S|$ is composed with a free rational pencil and $S|_S\sim C$, where $C$ is a general fiber of $\psi: W\to \Sigma'$.


\begin{lem}\label{lem: P2 bound pg3}
Let $X$ be a minimal $3$-fold of general type with $p_g(X)=3$ and $K_X^3=1$. Then $P_2(X)\le 7$. Moreover, the equality holds only if $|2K_X|$ defines a generically finite map of degree $2$.
\end{lem}
\begin{proof}
Recall that $\Phi_{|2K_X|}$ denotes the rational map defined by $|2K_X|$.
 Denote by $\Sigma_2$ the image of $\Phi_{|2K_X|}$. Since the canonical image of $X$ is $\mathbb{P}^2$ by \cite[Theorem~1.5(4)]{Chen07}, 
 it is clear that $\dim \Sigma_2\ge 2$. 

 If $\dim \Sigma_2=2$, then $|2K_X|$ induces a fibration $\psi_2: W\to \Sigma'_2$ which factors through $\psi: W\to \Sigma'$ and $ \Sigma'_2$ is birational to $\Sigma'$. In particular, we can identify a general fiber of 
 $\psi_2$ with $C$. 
 So we may write
${M_2}|_{M_2}\equiv d_2 C$,
where $M_2$ is a general element in $\Mov|2\pi^*(K_X)|$ and $d_2\ge P_2(X)-2$. Then
\begin{align*}
 4K_X^3\ge (\pi^*(K_X)\cdot M_2\cdot M_2)\ge d_2(\pi^*(K_X)\cdot C)\ge d_2\ge P_2(X)-2,
\end{align*}
 where the third inequality follows from \cite[\S3.2]{Chen07} as $(\pi^*(K_X)\cdot C)=\xi\geq 1$. We deduce that $P_2(X)\le 6$ when 
 $\dim\Sigma_2=2$.

If $\dim \Sigma_2=3$, then by \cite[Theorem~1.5(4)]{Chen07}, $\Phi_{|4K_X|}$ is not birational. So $\deg(\Phi_{|2K_X|})\ge 2$. 
Then 
 \[
 8K_X^3\ge \deg(\Phi_{|2K_X|})\deg (\Sigma_2)\ge 2(P_2(X)-3),
\]
 where $\deg(\Sigma_2)\ge P_2(X)-3$. We conclude that $P_2(X)\le 7$. Moreover, when $P_2(X)=7$, the above inequalities become equalities. In particular, 
 $\deg(\Phi_{|2K_X|})=2$. 
\end{proof}

\begin{lem}\label{lem: P3 bound pg3}
Let $X$ be a minimal $3$-fold of general type with $p_g(X)=3$ and $K_X^3=1$. Then
 $P_3(X)\ge P_2(X)+6$. 
\end{lem}
\begin{proof}
Kepp the notation in Setting~\ref{Set up for canonical and pluricanonical maps}.
 By the Kawamata--Viehweg vanishing theorem, we have
 \[
 |3K_W||_S\lsgeq |K_W+\roundup{\pi^*(K_X)}+S||_S\lsgeq |K_S+\roundup{\pi^*(K_X)|_S}|.
 \]
 Since $\pi^*(K_X)|_S\ge S|_S=C$, we have
\begin{align}
 |3K_W||_S\lsgeq |K_S+C|.\label{eq: 3K>K+C}
\end{align}
 By \cite[Theorem~3.4]{Horikawa1}, $h^1(\mathcal{O}_S)=0$. By the exact sequence
 \[
 0\to \mathcal{O}_S(K_S)\to \mathcal{O}_S(K_S+C)\to \mathcal{O}_C(K_C)\to 0,
 \]
 we deduce that $h^0(S, K_S+C)=p_g(S)+g(C)\geq 6$. Thus \eqref{eq: 3K>K+C} implies that
 \[
 P_3(X)\geq h^0(3K_W-S)+h^0(S, K_S+C)\ge P_2(X)+6.
 \]
 The proof is completed.
\end{proof}
\begin{prop}\label{prop: RR X12}
Let $X$ be a minimal $3$-fold of general type with $p_g(X)=3$ and $K_X^3=1$. Then
\begin{enumerate}
\item $\chi(\mathcal{O}_X)=-2$, $P_2(X)=7$;
 \item the basket of $X$ is $B_X=\{2\times (1, 2)\}$;
 \item the Hilbert series of $X$ is given by
\[
\sum_{k\geq 0}h^0(X, kK_X)q^k=\frac{1-q^{12}}{(1-q)^3 (1-q^{2})(1-q^{6})}.
\]
\end{enumerate}
All the above data are the same as those of a general hypersurface of degree $12$ in $\mathbb{P}(1,1,1,2,6)$.
\end{prop}
\begin{proof}
 By Lemma~\ref{lem: singularity of X pg3} and \cite[(6.1), (6.4)]{Rei87}, the basket $B_X$ of $X$ is $\{a\times (1,2)\}$ for some $a\geq 0$. By Reid's Riemann--Roch formula \eqref{eq: RR} and the Kawamata--Viehweg vanishing theorem, we have
\begin{align*}
 P_2(X){}&=\frac{1}{2}-3\chi(\mathcal{O}_X)+\frac{a}{4},\\
 P_3(X){}&=\frac{5}{2}-5\chi(\mathcal{O}_X)+\frac{a}{4}.
\end{align*}
Then by Lemmas~\ref{lem: P3 bound pg3} and \ref{lem: P2 bound pg3}, 
\[
14\geq 2P_2(X)=3P_3(X)-3P_2(X)-5+\frac{a}{2}\geq 13+\frac{a}{2}.
\]
So $P_2(X)=7$ and $a\leq 2$.
Then it is not hard to see that $a=2$ and $\chi(\mathcal{O}_X)=-2$ from the equation of $P_2(X)$.
Assertions (1) and (2) are proved.

Note that Reid's Riemann--Roch formula only depends on $K_X^3$, $\chi(\mathcal{O}_X)$ and $B_X$. 
These data are the same as those of a general hypersurface of degree $12$ in $\mathbb{P}(1,1,1,2,6)$ (see \cite[Table~3, No.~7]{Fle00}). 
So the Hilbert series of them also coincide and can be computed by \cite[Theorem~3.4.4]{WPS}.
\end{proof}
\begin{prop}\label{prop: |mK| pg3}
Let $X$ be a minimal $3$-fold of general type with $p_g(X)=3$ and $K_X^3=1$. Then
\begin{enumerate}
\item $|K_X|$ is not composed with a pencil; 
\item $|2K_X|$ defines a generically finite map of degree $2$;
\item $|6K_X|$ defines a birational map.
\end{enumerate}
\end{prop}
\begin{proof}
Assertion (1) follows from 
\cite[Theorem~1.5(4)]{Chen07}. Assertion (2) follows from Lemma~\ref{lem: P2 bound pg3} and $P_2(X)=7$ in Proposition~\ref{prop: RR X12}. Assertion (3) follows from \cite[Theorem~1.2]{Chen03}.
\end{proof}

\section{Birational geometry of $3$-folds with $p_g=4$ and $\Vol=2$}\label{sec: pg4}
In this section, we study the birational geometry of a minimal $3$-fold $X$ of general type with $p_g(X)=4$ and $K_X^3=2$. 

Recall that $\Sigma$ denotes the canonical image of $X$ in $\mathbb{P}^{3}$.
By \cite[Theorem~1.5(5)]{Chen07}, we have $\dim\Sigma\ge 2$.

When $\dim\Sigma=3$, the geometry of $X$ is well-understood by Kobayashi \cite{Kob}. We summarize the result with more detail as the following.

\begin{prop}\label{prop: X pg4 1}
 Let $X$ be a minimal $3$-fold of general type with $p_g(X)=4$ and $K_X^3=2$. Suppose that the canonical map of $X$ is generically finite. Then
 \begin{enumerate}
 \item $X$ is Gorenstein; 
 \item the canonical map is a generically finite morphism to $\mathbb{P}^3$ of degree $2$;

 \item the canonical model of $X$ is a hypersurface of degree $10$ in $\mathbb{P}(1, 1, 1, 1, 5)$;

\item the Hilbert series of $X$ is given by
 \[
\sum_{k\geq 0}h^0(X, kK_X)q^k=\frac{1-q^{10}}{(1-q)^4(1-q^{5})}.
\]
 \end{enumerate} 
\end{prop}
\begin{proof}
 Assertions (1) and (2) are due to \cite[Proposition~2.5]{Kob}. 
Denote by $Y$ the canonical model of $X$ and $\phi: Y\to \mathbb{P}^3$ the canonical map which is a morphism defined by $|K_Y|$. Then $\phi$ is finite as $K_Y$ is ample. Moreover, $\deg (\phi)=2$.
Denote by $D\subseteq \mathbb{P}^3$ the branch locus, then $K_Y=\phi^*(K_{\mathbb{P}^3}+\frac{1}{2}D)$ by the ramification formula. By computing the self-intersection number, we get $ D\in |\mathcal{O}_{\mathbb{P}^3}(10)|$. So $Y$ can be expressed as a hypersurface of degree $10$ in $\mathbb{P}(1, 1, 1, 1, 5)$ defined by $ x_4^2+f_0(x_0, x_1, x_2, x_3)=0$, where $[x_0:\dots: 
x_4]$ are coordinates of $\mathbb{P}(1, 1, 1, 1, 5)$ and $f_0$ is the defining equation of $D\subseteq \mathbb{P}^3$. 
Finally, to compute the Hilbert series, we can use either \cite[3.4.2]{WPS} or Reid's Riemann--Roch formula \eqref{eq: RR}. Here recall that $K_X^3=2$ and $B_X=\emptyset$, and $\chi(\OO_X)=-3$ which can be computed from the double cover construction. 
\end{proof}

Thus we only need to treat the case when $\dim\Sigma=2$. We recall the following properties from \cite[\S3.2]{Chen07}.

\begin{lem}\label{lem: pg 4 chen}
 Let $X$ be a minimal $3$-fold of general type with $p_g(X)=4$ and $K_X^3=2$. Suppose that the canonical map of $X$ is not generically finite. Let $\pi: W\to X$ be as in Setting~\ref{Set up for canonical and pluricanonical maps} and keep the same notation. 
 Denote by $C$ a general fiber of $\psi: W\to \Sigma'$. Then 
 \begin{enumerate}
 \item $(\pi^*(K_X)\cdot C)=1$;
 \item $C$ is a smooth projective curve of genus $2$;
 \item $S|_S\sim 2C$ which is a rational pencil;
 \item $\Sigma'\simeq \Sigma\subseteq \mathbb{P}^3$ is a nondegenerate surface of degree $2$, in particular, $\Sigma$ is either $\mathbb{P}(1,1,2)$ or $\mathbb{P}^1\times \mathbb{P}^1$.
 \end{enumerate}

\end{lem}

\begin{proof}
Assertions (1)--(3) are directly from \cite[\S3.2]{Chen07}. 
 For Assertion (4), just note that $S|_S\equiv\deg(s)\deg(\Sigma)C$ and $\deg(\Sigma)\geq 2$ as it is nondegenerate. So $\deg(\Sigma)=2$ and $\deg (s)=1$ by Assertion (3). By \cite[\S10]{Nagata}, $\Sigma$ is either $\mathbb{P}(1,1,2)$
or $\mathbb{P}^1\times \mathbb{P}^1$. In particular, $\Sigma$ is normal and hence $s$ is an isomorphism.
\end{proof}

By adopting the idea in \cite[Proof of Proposition~2.1]{HZ}, we have the following lemma.

\begin{lem}\label{lem: free pencil pg 4}
 Let $X$ be a minimal $3$-fold of general type with $p_g(X)=4$ and $K_X^3=2$. Suppose that the canonical map of $X$ is not generically finite. Then
 there exits a minimal $3$-fold $X_1$ birational to $X$, admitting a fibration $f_1: X_1\to \mathbb{P}^1$ whose general fibers are $(1,2)$-surfaces.
\end{lem}
\begin{proof}
By Lemma~\ref{lem: pg 4 chen}, $\Sigma$ is either $\mathbb{P}(1,1,2)$ or $\mathbb{P}^1\times \mathbb{P}^1$. 

\medskip

{\bf Case 1.} $\Sigma\simeq \mathbb{P}(1,1,2)$.

In this case, there is a contraction $r: \mathbb{F}_2\to \mathbb{P}^3$ from the second Hirzebruch surface induced by the linear system $|s+2\ell|$ such that $\Sigma=r(\mathbb{F}_2)$. Here $\ell$ is the ruling of the natural fibration $p: \mathbb{F}_2\to \mathbb{P}^1$ and $s$ is the unique negative section. 

Replacing $W$ by its birational model, we may assume that there is a morphism $\varphi: W\to \mathbb{F}_2$ such that $\psi=r\circ\varphi$. Thus we obtain a fibration $f_W:=p\circ\varphi: W\to \mathbb{P}^1$ with a general fiber $F_W=\varphi^*(\ell)$. Let $\zeta: W\dashrightarrow X_1$ be the contraction to a relative minimal model over $\mathbb{P}^1$, where we have a relatively minimal fibration $f_1: X_1\to\mathbb{P}^1$ with a general fiber denoted by $F_1$. Up to a further birational modification of $W$, we may assume that $\zeta$ is a morphism. Thus $\sigma_1:=\zeta|_{F_W}: F_W\to F_1$ is just the contraction to its minimal model. 
We claim that $F_1$ is a $(1,2)$-surface and $X_1$ is minimal.

Note that by definition, 
\[
\pi^*(K_X)\geq \psi^*\mathcal{O}_\Sigma(1)\geq \varphi^*(2\ell)=2F_W.
\]
By \cite[Corollary~2.3]{Noether}, $\pi^*(K_X)|_{F_W}-\frac{2}{3}\sigma_1^*(K_{F_1})$ is $\mathbb{Q}$-effective. Note that $S|_{F_W}\sim C$. By Lemma~\ref{lem: pg 4 chen}(1), we have
\[
(\sigma_1^*(K_{F_1})\cdot C)\le \frac{3}{2} (\pi^*(K_X)|_{F_W}\cdot C)=\frac{3}{2}.
\]
Then $(\sigma_1^*(K_{F_1})\cdot C)=1$. Since $C$ is a moving curve on $F_W$, we conclude that $F_1$ is a $(1,2)$-surface by \cite[Lemma~2.4]{EXPIII}. Note that $K_{F_W}\ge \pi^*(K_X)|_{F_W}$ and by Lemma~\ref{lem: pg 4 chen}(1),
\[
(\pi^*(K_X)|_{F_W})^2\ge (\pi^*(K_X)|_{F_W}\cdot S|_{F_W})=(\pi^*(K_X)|_{F_W}\cdot C)=1=(\sigma_1^*(K_{F_1}))^2.
\]
We deduce that $\pi^*(K_X)|_{F_W}=\sigma_1^*(K_{F_1})$ 
by considering the Zariski decomposition of $K_{F_W}$ and Lemma~\ref{lem: HIT}(1). 
Thus $X_1$ is minimal by \cite[Lemma~3.2]{Noether}. We finish the proof in this case.

\medskip

{\bf Case 2.} $\Sigma\simeq \mathbb{P}^1\times \mathbb{P}^1$.

Denote by $p_1$ and $p_2$ the two natural projections of $\mathbb{P}^1\times\mathbb{P}^1$. For $i=1, 2$, write $f_{W,i}=p_i\circ \gamma: W\to \mathbb{P}^1$. Then $f_{W,i}$ is a fibration for each $i$. Let $T_i$ be a general fiber of $f_{W, i}$. Then 
\begin{align}
\pi^*(K_X)\ge \psi^*(\mathcal{O}_\Sigma(1))=T_1+T_2.\label{eq: K>T+T}
\end{align}
Since $K_X^3=2$, we may assume that $(\pi^*(K_X)|_{T_1})^2\le 1$. Same as in {Case 1}, we get a contraction $\zeta: W\to X_1$ to a relative minimal model of $f_{W,1}$, where we denote by $f_1: X_1\to \mathbb{P}^1$ the relatively minimal fibration and by $F_1$ its general fiber. Thus $\sigma_1:=\zeta|_{T_1}: T_1 \to F_1$ is just the contraction to its minimal model. We claim that $F_1$ is a $(1,2)$-surface and $X_1$ is minimal.

For sufficiently large $m$, we have
\[
|2mK_W||_{T_1}\lsgeq|m(K_W+T_1+T_2)||_{T_1}\lsgeq|m(K_{T_1}+{T_2}|_{T_1})|,
\]
where the first inequality is by \eqref{eq: K>T+T} and the second is by \cite[Theorem~2.4(2)]{CJ}.
Hence by considering the movable parts of both sides, we deduce that
 $\pi^*(2K_X)|_{T_1}- \sigma_1^*(K_{F_1})- {T_2}|_{T_1}$
is $\mathbb{Q}$-effective. 
Note that $T_2|_{T_1}=C$. Combining with Lemma~\ref{lem: pg 4 chen}(1) and taking the intersection with $\pi^*(K_X)|_{T_1}$, we get
\[
(\pi^*(K_X)|_{T_1}\cdot \sigma_1^*(K_{F_1}))\le 2 (\pi^*(K_X)|_{T_1})^2-(\pi^*(K_X)|_{T_1}\cdot C)\le 1.
\]
Since $\pi^*(K_X)|_{T_1}\ge C$ by \eqref{eq: K>T+T}, we have $(C\cdot \sigma_1^*(K_{F_1}))\le 1$. By \cite[Lemma~2.4]{EXPIII}, $F_1$ is a $(1,2)$-surface. 
Furthermore, by Lemma~\ref{lem: pg 4 chen}(1),
\[
1\ge (\pi^*(K_X)|_{T_1})^2\ge (\pi^*(K_X)|_{T_1}\cdot C)=1.
\]
Thus $(\pi^*(K_X)|_{T_1})^2=1$. By considering the Zariski decomposition of $K_{T_1}$ and Lemma~\ref{lem: HIT}(1), we have
$\pi^*(K_X)|_{T_1}=\sigma_1^*(K_{F_1})$. 
By \cite[Lemma~3.2]{Noether}, $X_1$ is minimal. The proof is completed.
\end{proof}

Now we can understand the birational geometry in the case that the canonical map is not generically finite (cf. Proposition~\ref{prop: X pg4 1}).

\begin{prop}\label{prop: singularity of X pg4}
 Let $X$ be a minimal $3$-fold of general type with $p_g(X)=4$ and $K_X^3=2$. Suppose that the canonical map of $X$ is not generically finite. Then 
 \begin{enumerate}
 \item $h^1(\mathcal{O}_X)=h^2(\mathcal{O}_X)=0$; in particular, $\chi(\mathcal{O}_X)=-3$;
 \item $X$ is Gorenstein;
 \item the Hilbert series of $X$ is given by
 \[
\sum_{k\geq 0}h^0(X, kK_X)q^k=\frac{1-q^{10}}{(1-q)^4(1-q^{5})};
\]

 \item 
 $|2K_X|$ defines a generically finite map of degree $2$;

 \item {$|5K_X|$ defines a birational map.}
 \end{enumerate}
\end{prop}
\begin{proof}
 By Lemma~\ref{lem: free pencil pg 4}, after replacing by a birational model, we may assume that $X$ admits a fibration $f: X\to \mathbb{P}^1$ with a general fiber $F$ a $(1,2)$-surface. Now we are in the setting of \cite[\S3]{HZ}. Assertion (1) follows from \cite[Lemma~3.4]{HZ}.

Recall Reid's Riemann--Roch formula \eqref{eq: RR} for $P_2(X)$:
\[
P_2(X)=\frac{1}{2}K_X^3-3\chi(\mathcal{O}_X)+l_2(X).
\]
Here $l_2(X)\ge 0$ and $l_2(X)=0$ if and only if $X$ is Gorenstein. So $P_2(X)\ge 10$.
On the other hand,
since $K_X^3<\frac{4}{3}p_g(X)-\frac{17}{6}$, by \cite[Proposition~3.7]{HZ}, we have
\[
P_2(X)\le \rounddown{2K_X^3}+\rounddown{2K_X^3-\frac{5(p_g(X)-1)}{3}}+7= 10.
\]
So $P_2(X)=10$. Thus $l_2(X)=0$ and $X$ is Gorenstein. Assertion (2) is proved.

For Assertion (3), note that Reid's Riemann--Roch formula only depends on $K_X^3$, $\chi(\mathcal{O}_X)$ and $B_X$. 
These data are the same as those of a general hypersurface of degree $10$ in $\mathbb{P}(1,1,1,1,5)$ (see \cite[Table~3, No.~5]{Fle00}). 
So the Hilbert series of them also coincide and can be computed by \cite[Theorem~3.4.4]{WPS}.

For Assertion (4), note that the inequality in \cite[Proposition~3.7(1)]{HZ} is an equality now. The inequality in {Step 1} of \cite[Proof of Proposition~3.7]{HZ} is also an equality. In particular, $|2K_W||_F$ defines a generically finite map of degree at most $2$. Thus $|2K_X|$ defines a generically finite map of degree at most $2$. However, 
$|2K_X|$ does not define a birational map as $|2K_F|$ does not define a birational map on a $(1,2)$-surface $F$ by \cite[Lemma~2.1]{Horikawa}. So $|2K_X|$ defines a generically finite map of degree $2$.

Assertion (5) follows from \cite[Theorem~1.2]{Chen03}.
\end{proof}

\section{Characterization of weighted hypersurfaces}\label{sec: wt embed criterion}

The following theorem is a refinement of \cite{330-2}. It provides a criterion for embedding a polarized variety into a weighted projective space by numerical and geometric properties of this variety. 
In the statement, (1)(2) are 
numerical conditions which are usually easy to get, but (3)(4)(5) are geometric conditions which in general require hard work to get. 

\begin{thm}\label{thm: weighted embedding}
 Let $Y$ be a normal projective $3$-fold and let $H$ be a $\mathbb{Q}$-Cartier Weil ample divisor on $Y$. Suppose that there exists a well-formed quasismooth weighted hypersurface \[X_{d}\subseteq \mathbb{P}(1,a_1,a_2,a_3,a_4)\] with $1\leq a_1\leq a_2\leq a_3<a_4$ and $d=2a_4$ 
 such that 
 \begin{enumerate}
 \item $H^3=(\mathcal{O}_{X_d}(1))^3=\frac{2}{a_1a_2a_3}$ and 
 \item 
 $h^0(Y, kH)=h^0(X_d, \mathcal{O}_{X_d}(k))$ for any nonnegative integer $k$. 
 \end{enumerate}
 Suppose further that
 \begin{enumerate}[resume]
 \item $|a_2H|$ is not composed with a pencil;
 
 \item $|a_3H|$ defines a generically finite map, while either the degree of this map is greater than $1$, or $Y$ is not rational;
 \item $|a_4H|$ defines a birational map.
 \end{enumerate}
Then $Y$ is a hypersurface in $\mathbb{P}(1,a_1,a_2,a_3,a_4)$ defined by a weighted homogeneous polynomial $f$ of degree $d$, where 
\[f(x_0, \dots,
x_4)=x_4^2+f_0(x_0, x_1, x_2, x_3)\] in suitable homogeneous coordinates $[x_0:\dots:
x_4]$ of $\mathbb{P}(1,a_1,a_2,a_3,a_4)$; in other words, there is an isomorphism between graded $\mathbb{C}$-algebras
\[
\bigoplus_{k\geq 0}H^0(Y, kH)\simeq \mathbb{C}[x_0,\dots,
x_4]/(f),
\]
where $\mathbb{C}[x_0, \dots,
x_4]$ is viewed as a weighted polynomial ring with \[\wt(x_0, \dots,
x_4)=(1,a_1,a_2,a_3,a_4).\]
\end{thm}

\begin{rem}
In Theorem~\ref{thm: weighted embedding},
 by \cite[3.4.2]{WPS}, we have the following formula:
\[
\sum_{k\geq 0}h^0(X_d, \mathcal{O}_{X_d}(k)) q^k= \frac{1-q^{d}}{(1-q)(1-q^{a_1})(1-q^{a_2})(1-q^{a_3})(1-q^{a_4})}.
\]
We will always use this formula to compute $h^0(Y, kH)$. In practice, we do not need to calculate the precise value of $h^0(Y, kH)$, we will often identify the value with the cardinality of a certain set in a combinatorial way instead.
\end{rem}

 The following lemma allows us to describe generators and relations in certain linear systems. 
\begin{lem}\label{lem: fghp}
Keep the notation in Theorem~\ref{thm: weighted embedding}.
For $0\leq i\leq 4$, take distinct general elements
 $g_i\in H^0(Y, a_iH)\setminus\{0\}$ (where $a_0=1$). 
For a positive integer $k$, denote 
\[
S_k=\{g_0^{s_0}g_1^{s_1}g_2^{s_2}g_3^{s_3}\mid s_0, \dots, s_3\in \mathbb{Z}_{\geq 0}, s_0+a_1s_1+a_2s_2+a_3s_3=k\}.
\]
Then 
\begin{enumerate}
 \item the set $\{g_0, g_1, g_2, g_3\}$ is algebraically independent in the graded algebra $R(Y, H)=\bigoplus_{k\geq 0}H^0(Y, kH);$

\item $H^{0}(Y, a_4H)$ is spanned by a basis $S_{a_4}\cup\{g_4\}.$
\end{enumerate}
\end{lem}

Here we remark that $R(Y, H)$ is viewed as a subalgebra of $\mathbb{C}(Y)[u]$, where the variable $u$ is from the grading of $R(Y, H)$. So a rational function $g\in H^{0}(Y, kH)\subseteq \mathbb{C}(Y)$ is identified with $g u^{k}\in R(Y, H)$. But to simplify the notation, we will just write $g\in R(Y, H)$ as long as the grading of $g$ is clear in the context. 

\begin{proof}
First we show the following claim. 
\begin{claim}\label{claim: fgh}
 The set $\{g_0, g_1, g_2\}$ is algebraically independent in $R(Y, H)$.
\end{claim}
\begin{proof}
 
As $h^0(Y, a_1H)\geq 2$, $\{g_0, g_1\}$ is algebraically independent. 
By assumption, $|a_2H|$ is not composed with a pencil. So by the generality of $g_2$, the transcendental degree of $\mathbb{C}(g_0, g_1, g_2)$ as a subfield of the fractional field of $R(Y, H)$ is greater than $2$, which implies that $\{g_0, g_1, g_2\}$ is algebraically independent.
\end{proof}
By Claim~\ref{claim: fgh}, 
\[
S_k'=\{g_0^{s_0}g_1^{s_1}g_2^{s_2}\mid s_0, s_1, s_2\in \mathbb{Z}_{\geq 0}, s_0+a_1s_1+a_2s_2=k\}
\]
is linearly independent in $H^0(Y, kH)$ for any positive integer $k$. 
On the other hand, 
\[h^0(Y, a_3H)=h^0(X_d, \mathcal{O}_{X_d}(a_3))=|S_{a_3}|=|S'_{a_3}|+1.\]
 Then $S_{a_3}=S'_{a_3}\cup \{g_3\}$ is a basis of $H^0(Y, a_3H)$ by the generality of $g_3$. 
By assumption, $|a_3H|$ defines a generically finite map, so the transcendental degree of $\mathbb{C}(g_0, g_1, g_2, g_3)$ as a subfield of the fractional field of $R(Y, H)$ is greater than $3$, which implies that $\{g_0, g_1, g_2, g_3\}$ is algebraically independent. This proves Assertion (1).

In particular, $S_k$ is linearly independent in $H^0(Y, kH)$ for any positive integer $k$. So 
Assertion (2) follows from the computation that \[h^0(Y, a_4H)=h^0(X_d, \mathcal{O}_{X_d}(a_4))=|S_{a_4}|+1\] and the generality of $g_4$.
\end{proof}

 \begin{proof}[Proof of Theorem~\ref{thm: weighted embedding}]
Keep the notation in Lemma~\ref{lem: fghp}.
We can define $3$ rational maps: 
\begin{align*}
\Psi_{a_2}: {}&Y\dashrightarrow \mathbb{P}(1, a_1, a_2); \\
{}& P\mapsto [g_0(P):g_1(P):g_2(P)];\\
\Psi_{a_3}: {}&Y\dashrightarrow \mathbb{P}(1, a_1, a_2, a_3); \\
{}& P\mapsto [g_0(P):g_1(P):g_2(P):g_3(P)];\\
\Psi_{a_4}: {}&Y\dashrightarrow \mathbb{P}(1, a_1, a_2, a_3, a_4);\\
{}& P\mapsto [g_0(P):\dots:
g_4(P)].
\end{align*}
We claim that they have the following geometric properties.
\begin{claim}\label{claim: claim2}
\begin{enumerate}

\item $\Psi_{a_2}$ is dominant; $\Psi_{a_3}$ is dominant and generically finite of degree $2$;
\item $\Psi_{a_4}$ is birational onto its image;
\item let $Y'$ be the closure of $\Psi_{a_4}(Y)$ in $\mathbb{P}(1, a_1, a_2, a_3, a_4)$, then $Y'$ is defined by a weighted homogeneous polynomial $f$ of degree $d$, where 
\[f(x_0, \dots,
x_4)=x_4^2+f_0(x_0, x_1, x_2, x_3)\] in suitable homogeneous coordinates $[x_0:\dots:
x_4]$ of $\mathbb{P}(1, a_1, a_2, a_3, a_4)$.
\end{enumerate}\end{claim}
\begin{proof}
(1) By Lemma~\ref{lem: fghp}, $\{g_0, g_1, g_2, g_3\}$ is algebraically independent, so $\Psi_{a_2}$ and $\Psi_{a_3}$ are dominant. In particular, $\Psi_{a_3}$ is generically finite by dimension reason.
The degree of $\Psi_{a_3}$ is the number of points in the fiber over a general point in $\mathbb{P}(1, a_1, a_2, a_3)$. After taking a resolution $\pi: W\to Y$ such that 
$N_{k}=\Mov|\rounddown{\pi^*(a_kH)}|$ is free for $k=1,2,3$, the degree is just the intersection number $(N_1\cdot N_2\cdot N_{3})$. So 
\[
\deg (\Psi_{a_3})=(N_1\cdot N_2\cdot N_{3})\leq a_1a_2a_3H^3=2.
\]
 On the other hand, it is clear that $\deg(\Psi_{a_3})>1$ by the assumption that either $|a_3H|$ defines a generically finite map of degree greater than $1$ or $Y$ is not rational.
 
 \medskip
 
(2) By assumption, $|a_4H|$ defines a birational map. As $g_4$ is general, it can separate two points in a general fiber of $\Psi_{a_3}$, so $\Psi_{a_4}$ is birational onto its image.

\medskip

(3) Note that $d=2a_4$ and $h^0(Y, dH)=|S_{2a_4}|+|S_{a_4}|$.
 On the other hand, 
\[S_{2a_4}\sqcup (S_{a_4}\cdot g_4)\sqcup \{g_4^2\}\subseteq H^0(Y, dH).\] 
So $S_{2a_4}\sqcup (S_{a_4}\cdot g_4)\sqcup \{g_4^2\}$ is linearly dependent in $H^0(Y, dH)$. In other words, there exists a weighted homogeneous polynomial $f(x_0,\dots,
x_4)$ of degree $d$ with $\text{wt}(x_0, \dots,
x_4)=(1,a_1,a_2,a_3,a_4)$ such that 
$f(g_0, \dots,
g_4)=0.$
So $Y'$ is contained in $(f=0)\subseteq \mathbb{P}(1, a_1,a_2,a_3,a_4)$. Note that $Y'$ is a hypersurface in $\mathbb{P}(1, a_1,a_2,a_3,a_4)$ by dimension reason.

We claim that $Y'=(f=0)$ and $x_4^2$ has nonzero coefficient in $f$.
Otherwise, either $Y'$ is defined by a weighted homogeneous polynomial of degree at most $d-1$, or $x_4^2$ has zero coefficient in $f$. In either case,
$Y'$ is defined by a weighted homogeneous polynomial $\tilde{f}$ of the form 
\[\tilde{f}(x_0, \dots,
x_4)=x_4\tilde{f}_1(x_0, x_1, x_2, x_3)+\tilde{f}_2(x_0, x_1, x_2, x_3).\]
Here $\tilde{f}_1\neq 0$ as $\{g_0, g_1, g_2, g_3\}$ is algebraically independent.
Then $Y'$ is birational to $\mathbb{P}(1,a_1,a_2,a_3)$ under the rational projection map
\begin{align*}
\mathbb{P}(1, a_1,a_2,a_3, a_4){}&\dashrightarrow \mathbb{P}(1, a_1,a_2,a_3);\\
[x_0:\dots :x_4]{}&\mapsto [x_0:\dots :x_3].
\end{align*}
But this contradicts the fact that $\deg (\Psi_{a_3})=2$.
So $Y'=(f=0)$ and $x_4^2$ has nonzero coefficient in $f$. After a suitable coordinate transformation, we may assume that $f=x_4^2+f_0(x_0, x_1, x_2, x_3)$. 
\end{proof}

Now go back to the proof of Theorem~\ref{thm: weighted embedding}.
By Claim~\ref{claim: claim2}, $f$ is the only algebraic relation on $g_0, \dots, 
g_4$. Denote $\mathcal{R}$ to be the graded sub-$\mathbb{C}$-algebra of $R(Y, H)$ 
generated by $\{g_0, \dots, 
g_4\}$. Then there is a natural isomorphism between graded $\mathbb{C}$-algebras
\[
\mathcal{R}\simeq \mathbb{C}[x_0,\dots,
x_4]/(x_4^2+f_0(x_0, x_1, x_2, x_3))
\]
by sending $g_i\mapsto x_i$ ($0\leq i\leq 4$) and the right-hand side is exactly the coordinate ring of $Y'$.
Write $\mathcal{R}=\bigoplus_{k\geq 0}\mathcal{R}_k$, where $\mathcal{R}_k$ is the homogeneous part of degree $k$. Then,
by \cite[3.4.2]{WPS}, 
\[
\sum_{k\geq 0}(\dim_\mathbb{C} \mathcal{R}_k)q^k= \frac{1-q^{d}}{(1-q)(1-q^{a_1})(1-q^{a_2})(1-q^{a_3})(1-q^{a_4})}.
\]
So $\dim_\mathbb{C}\mathcal{R}_k=h^0(X_d, \mathcal{O}_{X_d}(k))=h^0(Y, kH)$ for any $k\in \mathbb{Z}_{\geq 0}$, and hence the inclusion $\mathcal{R}\subseteq R(Y, H)$ is an identity.
 This implies that 
\[Y\simeq \Proj R(Y, H)=\Proj\mathcal{R}\simeq Y'. \]
This finishes the proof.
 \end{proof}

\section{Proofs of Theorem \ref{thm: pg=2}, Theorem \ref{thm: pg=3} and Theorem \ref{thm: pg=4}}\label{sec: proofs}

In this section, we prove Theorem \ref{thm: pg=2}, Theorem \ref{thm: pg=3} and Theorem \ref{thm: pg=4}. The proof will be divided into $4$ parts. 

\subsection{Plurigenera}
In this subsection, we determine the plurigenera.
\begin{proof}[Proof of Theorems~\ref{thm: pg=2}(1), \ref{thm: pg=3}(1), and \ref{thm: pg=4}(1)]
These are directly from Propositions~\ref{prop: RR X16}, \ref{prop: RR X12}, \ref{prop: X pg4 1}, and \ref{prop: singularity of X pg4}.
\end{proof}

\subsection{Weighted embedding}

In this subsection, we show that the varieties we consider can be embedded as weighted hypersurfaces by Theorem~\ref{thm: weighted embedding}.

\begin{proof}[Proof of Theorems~\ref{thm: pg=2}(2) and~\ref{thm: pg=3}(2)]
We may replace $W$ by its canonical model $Y$ and apply Theorem~\ref{thm: weighted embedding} to $H=K_Y$. 
Then Theorem~\ref{thm: pg=2}(2)
is a direct consequence of Propositions~\ref{prop: RR X16} and \ref{prop: |mK|}; Theorem~\ref{thm: pg=3}(2)
is a direct consequence of Propositions~\ref{prop: RR X12} and \ref{prop: |mK| pg3}.
\end{proof}
 
The proof of Theorem~\ref{thm: pg=4}(2) is slightly more complicated as Theorem~\ref{thm: weighted embedding} is not directly applicable.

\begin{proof}[Proof of Theorem~\ref{thm: pg=4}(2)]
We may replace $W$ by its canonical model $Y$.

Suppose that the canonical map of $Y$ is generically finite, then we get (2a) by Proposition~\ref{prop: X pg4 1}. 

Suppose that the canonical map of $Y$ is not generically finite. Then we are going to embed $Y$ into $\mathbb{P}(1,1,1,1,2,5)$. Theorem~\ref{thm: weighted embedding} is not directly applicable here as we need $2$ defining equations, but the proof is similar to that of Theorem~\ref{thm: weighted embedding} with some essential modification. 

Since $p_g(Y)=4$, we may take $\{g_0, g_1, g_2, g_3\}$ as a basis of $H^0(Y, K_Y)$. Take general elements 
$g_4\in H^0(Y, 2K_Y)$ and $g_5\in H^0(Y, 5K_Y)$. 
We regard $g_i$ as elements in $R(Y)=\bigoplus_{k\geq 0}H^0(Y, kK_Y)$.

By Lemma~\ref{lem: free pencil pg 4}, the canonical image of $Y$ in $\mathbb{P}^3$ is a surface of degree $2$, so $g_0, g_1, g_2, g_3$ satisfy a unique quadratic algebraic relation $q(g_0, g_1, g_2, g_3)=0$. Without loss of generality, we may assume that $q$ is of the form
\[
q(x_0, x_1, x_2, x_3)=x_3^2+q_0(x_0, x_1, x_2).
\]
In particular, $\{g_0, g_1, g_2\}$ is algebraically independent and $\mathbb{C}(g_0, \dots, g_3)$ is a field extension of $\mathbb{C}(g_0, g_1, g_2)$ of degree $2$. 
Then, for any $k\geq 0$,
\[
S_k=\{g_0^{s_0}g_1^{s_1}g_2^{s_2}g_3^{s_3}\mid s_0, \dots, s_3\in \mathbb{Z}_{\geq 0}, s_0+ s_1+ s_2+ s_3=k \text{ and } s_3\leq 1\}
\]
is linearly independent in $H^0(Y, kK_Y)$. 
By Proposition~\ref{prop: singularity of X pg4}(3),
\[
H^0(Y, 2K_Y)=10=|S_2|+1, 
\]
so $S_2\sqcup \{g_4\}$ forms a basis of $H^0(Y, 2K_Y)$. 
Since $|2K_Y|$ defines a generically finite map of degree $2$ by Proposition~\ref{prop: singularity of X pg4}(4), $\{g_0, g_1, g_2, g_4\}$ is algebraically independent.
Then $S_5\sqcup (S_3\cdot g_4)\sqcup (S_1\cdot g_4^2)$ is linearly independent in $H^0(Y, 5K_Y)$. 
By Proposition~\ref{prop: singularity of X pg4}(3),
\[
H^0(Y, 5K_Y)=57=|S_5|+|S_3|+|S_1|+1, 
\]
hence $S_5\sqcup(S_3\cdot g_4)\sqcup (S_1\cdot g_4^2)\sqcup \{g_5\}$ forms a basis of $H^0(Y, 5K_Y)$.

Now we can consider $2$ rational maps: 
\begin{align*}
\Psi_{2}: {}&Y\dashrightarrow \mathbb{P}(1,1,1,1,2); \\
{}& P\mapsto [g_0(P): \dots : 
g_4(P)];\\
\Psi_{5}: {}&Y\dashrightarrow \mathbb{P}(1,1,1,1,2,5); \\
{}& P\mapsto [g_0(P):\dots : 
g_5(P)].
\end{align*}
Then $\Psi_{2}$ is generically finite of degree $2$ as $|2K_Y|$ defines a generically finite map of degree $2$; moreover, as $|5K_Y|$ defines a birational map by Proposition~\ref{prop: singularity of X pg4}(5), $g_5$ separates fibers of $\Psi_{2}$ and hence $\Psi_{5}$ is birational onto its image.

By Proposition~\ref{prop: singularity of X pg4}(3),
\[
h^0(Y, 10K_Y)=342=\sum_{i=0}^5|S_{10-2i}|+\sum_{j=0}^2|S_{5-2j}|.
\]
On the other hand, 
\[\bigsqcup_{i=0}^5(S_{10-2i}\cdot g_4^i)\sqcup \bigsqcup_{j=0}^2(S_{5-2j}\cdot g_4^jg_5)\sqcup \{g_5^2\}\subseteq H^0(Y, 10K_Y).\]
So there is
 a weighted homogeneous polynomial $f$ of degree $10$ such that
$f(g_0, \dots, g_5)=0$.

Let $Y'$ be the closure of $\Psi_{5}(Y)$ in $\mathbb{P}(1, 1, 1, 1, 2, 5)$. Then clearly $Y'\subseteq (q=f=0)$. We claim that $Y'=(q=f=0)$ and $x_5^2$ has nonzero coefficient in $f$.
Otherwise, $Y'$ is contained in $(q=\tilde{f}=0)$, where 
$\tilde{f}$ of the form 
\[\tilde{f}(x_0, \dots,
x_5)=x_5\tilde{f}_1(x_0, \dots,
x_4)+\tilde{f}_2(x_0, \dots,
x_4).\] 
Here $\tilde{f}$ is not divisible by $q$, which implies that $\tilde{f}_1\neq 0$. 
But then $Y'$ is birational to $\Psi_{2}(Y)$ 
 under the rational projection map
\begin{align*}
{}&\mathbb{P}(1, 1,1,1,2,5)\dashrightarrow \mathbb{P}(1, 1,1,1,2),
\end{align*}
which contradicts the fact that $\Psi_{2}$ is not birational.
So $Y'=(q=f=0)$ and $x_5^2$ has nonzero coefficient in $f$. After a suitable coordinate transformation, we may assume that $f=x_5^2+f_0(x_0, \dots,
x_4)$. 

Now $q$ and $f$ are the only algebraic relations on $g_0, \dots, 
g_5$. Denote $\mathcal{R}$ to be the graded sub-$\mathbb{C}$-algebra of $R(Y)$
generated by $\{g_0, \dots, 
g_5\}$. Then there is a natural isomorphism between graded $\mathbb{C}$-algebras
\[
\mathcal{R}\simeq \mathbb{C}[x_0, \dots,
x_5]/(q, f)
\]
by sending $g_i\mapsto x_i$ ($0\leq i\leq 5$) and the right-hand side is exactly the coordinate ring of $Y'$.
Write $\mathcal{R}=\bigoplus_{k\geq 0}\mathcal{R}_k$, where $\mathcal{R}_k$ is the homogeneous part of degree $k$. Then,
by \cite[3.4.2]{WPS}, 
\[
\sum_{k\geq 0}(\dim_\mathbb{C} \mathcal{R}_k)q^k= \frac{(1-q^{2})(1-q^{10})}{(1-q)^4(1-q^{2})(1-q^{5})}=\frac{1-q^{10}}{(1-q)^4(1-q^{5})}.
\]
So $\dim_\mathbb{C}\mathcal{R}_k=h^0(Y, kK_Y)$ for any $k\in \mathbb{Z}_{\geq 0}$ by Proposition~\ref{prop: singularity of X pg4}(3), and hence the inclusion $\mathcal{R}\subseteq R(Y)$ is an identity.
This implies that 
\[ Y\simeq \Proj R(Y) = \Proj\mathcal{R}\simeq Y'. \]
This finishes the proof of (2b).

Finally, for a $3$-fold $Y$ defined by $(q=f=0)$ in $\mathbb{P}(1, 1,1,1,2,5)$, we can view it as a specialization of $Y_t$ ($t\in\mathbb{C}^*$) which is defined by \[(q+tx_4=f+tg=0),\] where $g$ is a general weighted polynomial in $x_0,\dots, x_3, x_5$ of degree $10$.
In particular, $(g=0)\subseteq \mathbb{P}(1,1,1,1,5)$ is smooth. 
Then $Y_t$ turns out to be a $3$-fold defined by 
\[
f(x_0, x_1, x_2, x_3, -\frac{1}{t}q, x_5)+tg(x_0, x_1, x_2, x_3, x_5)=0
\]
in $\mathbb{P}(1,1,1,1,5)$ with coordinates $[x_0: x_1: x_2: x_3: x_5]$, which is actually smooth for general $t\neq 0$. 
\end{proof}

\subsection{simply-connectedness}
In this subsection, we prove the simply-connectedness.
\begin{proof}[Proof of Theorems~\ref{thm: pg=2}(3), \ref{thm: pg=3}(3), and \ref{thm: pg=4}(3)]
Let $W$ be as in the assumption. 
By (2) of each theorem, 
the canonical model $Y$ of $W$ is a specialization of general hypersurfaces in a certain weighted projective space. 
Namely, there is a flat morphism $\mathcal{Y}\to C$ to a smooth curve such that the central fiber $\mathcal{Y}_o$ at $o\in C$ is isomorphic to $Y$ and a general fiber $\mathcal{Y}_t$ for $t\neq o$ is a 
general weighted hypersurface.
By \cite[Lemma~2.8.1]{Kollar-shafa}, after shrinking $C$, we may assume that $\mathcal{Y}\setminus \mathcal{Y}_o\to C\setminus \{o\}$ is a topological fiber bundle. 
Note that $\mathcal{Y}$ is normal by \cite[Theorem~23.9]{Matsumura}. Hence by \cite[Theorem~2.12]{Kollar-shafa}, there is a surjective map $\pi_1(\mathcal{Y}_t)\to \pi_1(\mathcal{Y}_o)$ for general $t\in C$. 
Such $\mathcal{Y}_t$ is quasismooth by \cite{Fle00} and hence is simply-connected by \cite[Theorem~3.2.4]{WPS}. Hence $Y$ is also simply-connected. Since fundamental group is a birational invariant among projective klt varieties by \cite[Theorem~1.1]{Takayama}, we conclude the simply-connectedness of $W$.
\end{proof}

\subsection{Irreducibility and dimensions of moduli spaces}
In this subsection, we prove the irreducibility of $3$ moduli spaces and compute their dimensions.
\begin{proof}[Proof of Theorem~\ref{thm: pg=2}(4)]
 Write $\mathbb{P}=\mathbb{P}(1,1,2,3,8)$. By Theorem~\ref{thm: pg=2}(2), it is clear that $\mathcal{M}_{\frac{1}{3}, 2}$ is irreducible. 
It is unirational as it is dominated by an open subvariety of a rational variety (parametrized by coefficients). 
 Furthermore, 
 \[
 \dim \mathcal{M}_{\frac{1}{3}, 2}=h^0(\mathbb{P}, \mathcal{O}_{\mathbb{P}}(16))-1-\dim \Aut \mathbb{P}.
 \]
 By an easy computation, 
 $h^0(\mathbb{P}, \mathcal{O}_{\mathbb{P}}(16))=246$. By \cite[\S4]{Cox95} and \cite[Proposition~4.3]{Cox14}, we have
 \begin{align*}
 \dim \Aut (\mathbb{P})&=2h^0(\mathcal{O}_{\mathbb{P}}(1))+ h^0(\mathcal{O}_{\mathbb{P}}(2))+h^0(\mathcal{O}_{\mathbb{P}}(3))+h^0(\mathcal{O}_{\mathbb{P}}(8))-1\\
 &=56.
 \end{align*}
Therefore, $\dim \mathcal{M}_{\frac{1}{3},2}=189$.
\end{proof}

\begin{proof}[Proof of Theorem~\ref{thm: pg=3}(4)]
 Write $\mathbb{P}=\mathbb{P}(1,1,1,2,6)$. By Theorem~\ref{thm: pg=3}(2), it is clear that $\mathcal{M}_{1, 3}$ is irreducible. It is unirational as it is dominated by an open subvariety of a rational variety (parametrized by coefficients). Furthermore, 
 \[
 \dim \mathcal{M}_{1, 3}=h^0(\mathbb{P}, \mathcal{O}_{\mathbb{P}}(12))-1-\dim \Aut \mathbb{P}.
 \]
 By an easy computation, 
 $h^0(\mathbb{P}, \mathcal{O}_{\mathbb{P}}(12))=303$. By \cite[\S4]{Cox95} and \cite[Proposition~4.3]{Cox14}, we have
 \begin{align*}
 \dim \Aut (\mathbb{P})&=3h^0(\mathcal{O}_{\mathbb{P}}(1))+ h^0(\mathcal{O}_{\mathbb{P}}(2))+h^0(\mathcal{O}_{\mathbb{P}}(6)) -1 =66.
 \end{align*}
Therefore, $\dim \mathcal{M}_{1,3}=236$.
\end{proof}

\begin{proof}[Proof of Theorem~\ref{thm: pg=4}(4)]
 Write $\mathbb{P}=\mathbb{P}(1,1,1,1,5)$. By Theorem~\ref{thm: pg=4}(2), it is clear that $\mathcal{M}_{2, 4}$ is irreducible. It is unirational as it is dominated by an open subvariety of a rational variety (parametrized by coefficients). Furthermore, 
 \[
 \dim \mathcal{M}_{2, 4}=h^0(\mathbb{P}, \mathcal{O}_{\mathbb{P}}(10))-1-\dim \Aut \mathbb{P}.
 \]
 By an easy computation, 
 $h^0(\mathbb{P}, \mathcal{O}_{\mathbb{P}}(10))= 343$. By \cite[\S4]{Cox95} and \cite[Proposition~4.3]{Cox14}, we have
 \begin{align*}
 \dim \Aut (\mathbb{P})&=4h^0(\mathcal{O}_{\mathbb{P}}(1))+ h^0(\mathcal{O}_{\mathbb{P}}(5)) -1 = 72.
 \end{align*}
Therefore, $\dim \mathcal{M}_{2, 4}= 270$.
\end{proof}

\section*{\bf Acknowledgments}

 This work was supported by National Key Research and Development Program of China \#2023YFA1010600, \#2020YFA0713200, and NSFC for Innovative Research Groups \#12121001. Chen was supported by NSFC grant \#12071078. Hu was supported by NSFC grant \#12201397. Chen and Jiang are members of the Key Laboratory of Mathematics for Nonlinear Sciences, Fudan University. 

We are grateful to Jungkai A. Chen, Stephen Coughlan, Roberto Pignatelli, Mingchen Xia and Tong Zhang for discussions.

\end{document}